\documentclass[10pt]{article}
\usepackage[margin=2.5cm]{geometry}
\usepackage{float}
\usepackage{adjustbox}
\usepackage{booktabs}
\usepackage{xcolor}
\usepackage{amsmath}
\usepackage{amssymb}
\usepackage{caption}
\usepackage{amsthm}
\usepackage{subcaption}
\usepackage{mathtools}
\usepackage{comment}
\usepackage{soul}
\usepackage{enumerate}
\usepackage{hyperref}

\theoremstyle{definition}
\newtheorem{assumption}{Assumption}[section]
\newtheorem{definition}{Definition}[section]
\newtheorem{remark}{Remark}[section]
\newtheorem{theorem}{Theorem}[section]
\newtheorem{example}{Example}[section]
\newtheorem{corollary}{Corollary}[section]
\newtheorem{lemma}{Lemma}[section]
\newtheorem{proposition}{Proposition}[section]

\usepackage[
  backend=biber,
  style=numeric,
  natbib=true,
  url=false,
  isbn=false,
  doi=true,
  eprint=false,
  sortcites=true,
  giveninits=true,
]{biblatex}
\AtBeginBibliography{\setlength{\emergencystretch}{1em}}
\newcommand{\E}{\mathbb{E}}
\newcommand{\R}{\mathbb{R}}
\newcommand{\Rp}{\mathbb{R}_+}
\newcommand{\N}{\mathbb{N}}
\newcommand{\Prob}{\mathbb{P}}
\newcommand{\Ecal}{\mathcal{E}}
\newcommand{\Pcal}{\mathcal{P}}

\newcommand{\Fcal}{\mathcal{F}}
\newcommand{\Bcal}{\mathcal{B}}
\DeclareMathOperator*{\esssup}{ess\,sup}

\usepackage{authblk}

\author[1,2]{Sven Karbach}
\author[3,4,5]{Thomas K. Kloster\thanks{\texttt{tkk@econ.au.dk}}}

\affil[1]{Korteweg--de Vries Institute for Mathematics, University of Amsterdam}
\affil[2]{Institute for Informatics, University of Amsterdam}
\affil[3]{Department of Economics and Business Economics, Aarhus University}
\affil[4]{Department of Data Science and Analytics, BI Norwegian Business School}
\affil[5]{CoRE, Center for Research in Energy: Economics and Markets}

\title{Affine pure-jump Volterra fields}
\begin{document}
\maketitle

\begin{abstract}
We study a class of non-negative spatio-temporal random fields that exhibit self-exciting
clustering, which we refer to as affine pure-jump Volterra fields. They are defined via
stochastic integration of a Volterra kernel against a thinned Poisson random measure and the
thinning is such that the field admits a representation as a stochastic integral of the same kernel, but against a random measure whose compensator has a density that is pointwise affine in the field itself. This representation leads to affine transform formulas characterizing the Laplace transform of functionals of the field up to the solution of a deterministic non-linear
Volterra integral equation. Affine pure-jump Volterra random fields extend a non-negative finite-first-moment subclass of pure-jump
affine Volterra processes to the random-field setting. The framework includes non-negative finite-variation
ambit fields, marked Hawkes systems, and branching-type models as special cases, and we make these
connections explicit through examples.


\end{abstract}

\section{Introduction}
We study a class of positive spatio-temporal random fields $V = (V_t(x))_{t\in[0,T],\,x\in E}$ arising from stochastic space-time Volterra equations of the form
\begin{equation}\label{eq:SVE_intro}
V_t(x) = \phi(t,x) + \int_{(0,t)}\int_E\int_0^\infty\int_0^\infty K(t-s,x,y)\, z\, \mathbf{1}_{\{u \le \lambda_0 + \lambda_1 V_s(y)}\}\, N(ds,dy,dz,du),
\end{equation}
where $N$ is a marked Poisson random measure on $[0,T] \times E \times \Rp \times \Rp$ with intensity $ds \otimes m(dy) \otimes \ell(dz) \otimes du$. Here, $E$ is a Polish space carrying a finite measure~$m$, $\ell$ is a L\'evy measure on $\Rp$ with finite first moment, $K$ is a non-negative (possibly Volterra-singular) deterministic kernel, $\phi$ is a non-negative driving field, and $\lambda_0,\lambda_1\ge 0$ are non-negative coefficients. We refer to random fields of the form \eqref{eq:SVE_intro} as \emph{affine pure-jump Volterra fields}, where the affine structure arises from the thinning by $\mathbf{1}_{\{u \le \lambda_0 + \lambda_1 V_s(y)\}}$, which produces a marked jump measure whose predictable compensator is pointwise affine in $V$. More precisely, it is shown in Section~\ref{sec:affine_structure} that $V_t(x)$ can be written on the form
\begin{equation}\label{eq:intro_affine_representation}
V_t(x)=\phi(t,x)+{\int_{(0,t)}}\int_E\int_{0}^\infty K(t-s,x,y)\,z\,\mu (ds,dy,dz),
\end{equation}
where the predictable compensator of the marked random measure $\mu$ admits a density with respect to $dt\otimes m \otimes \ell$ of the form
\begin{equation}\label{eq:hawkes_intensity_intro}
    \Lambda_t(x) = \lambda_0+\lambda_1{V_t(x)}.
\end{equation}
When $\lambda_1=0$, the field in \eqref{eq:SVE_intro} reduces to a L\'evy-driven ambit field in the sense of \cite{Ambit}; see \cite{BenthVeraart2026} for a recent overview. In this case, one-point laws are infinitely divisible and their Fourier--Laplace transforms have L\'evy--Khintchine form, cf. \cite{Rosinski1989}, while the temporal dynamics need not be semimartingales or Markov processes. When $\lambda_1>0$, the compensator \eqref{eq:hawkes_intensity_intro} produces self-exciting jump clustering and our main result derives conditional Laplace functionals from a deterministic Riccati--Volterra equation. Setting $E=\lbrace e\rbrace$ recovers the corresponding non-negative, finite-first-moment subclass of pure-jump affine Volterra processes studied in \cite{AbiJaberLarssonPulido2019,BondiLivieriPulido2024}. Under additional regularity, the map $t\mapsto V_t(\cdot)$ can also be viewed as a function-valued process and compared with Hilbert space-valued L\'evy-type models \cite{BenthSgarra2024,BenthKruhner2023,BenthEyjolfsson2017} and measure-valued branching processes \cite{DawsonLi2012,Li2022}. In the overlapping non-negative finite-variation cases the transform formulas have the corresponding affine form, whereas the random-field formulation leads to a pointwise Volterra equation rather than an equation posed in a function space.

Random-field models arise, for example, in term-structure modelling, astrophysics, and medical imaging; see \cite{Goldstein2000,Kimmel2004,Kloster,,MedicalImaging,MedicalImaging2}. Our focus is on non-negative spatio-temporal quantities with clustered activity. Financial returns exhibit pronounced volatility clustering \cite{Cont2001}, while earthquake occurrences provide a canonical setting for spatio-temporal self-excitation \cite{Earthquakes}. Throughout, we restrict to scalar-valued fields $V_t(x)$, but note that affine processes on more general state spaces include Hilbert space-valued affine diffusions \cite{SchmidtTappeYu2020}, measure-valued affine diffusions \cite{CuchieroMeasureValuedAffine}, affine processes on positive semidefinite matrices \cite{CuchieroFilipovicMayerhoferTeichmann2011}, and pure-jump affine processes on positive Hilbert--Schmidt operators \cite{CoxKarbachKhedherPureJump2022}. The latter construction is purely discontinuous on the cone of positive self-adjoint Hilbert--Schmidt operators and we likewise focus on pure-jump dynamics. This permits a probabilistically strong construction, whereas preserving positivity under an affinely state-dependent space-time Gaussian driver would require a different mechanism. Our approach is related to \cite{Chong2017}, who considers general (non-affine) L\'evy-driven Volterra fields of the form
\[
Y_t(x)=\int_{0}^{t}\int_E K(t-s,x,y)\sigma(Y_s(y))\mu(dy,ds),
\]
where the non-linearity $\sigma$ is assumed pointwise Lipschitz continuous in the state $Y_t(x)$ and $\mu$ is an independently scattered and infinitely divisible random measure. The Lipschitz assumption is violated for the indicator function appearing in \eqref{eq:SVE_intro}, and we therefore must establish existence and uniqueness of a strong solution to \eqref{eq:SVE_intro} by a different contraction argument. 

\subsection{Description of the affine transform formula}
Once existence and uniqueness is established for a suitable class of kernel functions in Theorem~\ref{thm:strong_existence}, our main result is the affine transform formula for the Laplace functional of $V$. Related spatial Hawkes models admit Laplace-functional fixed-point equations; see, for example, \cite{BaarsLaevenMandjes2024}. Here we develop a predictable field construction and conditional Laplace transforms for non-negative Volterra kernels, allowing integrable temporal singularities and infinite jump activity under a finite first-moment condition. In Theorem~\ref{thm:spatial_transform}, we obtain a formula of the form
\begin{equation}\label{eq:intro_transform_formula}
\E\left[\exp\left(-\langle f,V_t\rangle\right)\right] = \exp\left(-\langle f,\phi_t\rangle - \int_0^t\int_E\left(\lambda_0+\lambda_1\phi_s(y)\right)F\bigl(\psi_f(t-s,y)\bigr)\,m(dy)\,ds\right),
\end{equation}
where $\langle f,V_t\rangle=\int_{E}f(y)V_t(y)m(dy)$, $F(q)=\int_0^\infty (1-e^{-qz})\ell (dz)$, and $\psi_f$ is the unique solution in the dominated class of Lemma~\ref{lem:riccati_volterra} of a deterministic non-linear Volterra equation with kernel $K$, non-linearity $F$, and a forcing term determined by $f$. Equation~\eqref{eq:intro_transform_formula} is the analogue of the usual affine transform formula and Theorem~\ref{thm:conditional_transform} gives the corresponding conditional transform, while Section~\ref{sec:spacetime_transform} treats functionals integrated over space and time.

The adjective \emph{affine} refers here to two related structures. First, the logarithm of \eqref{eq:intro_transform_formula} depends affinely on the driving field $\phi$, because the Riccati--Volterra solution $\psi_f$ is independent of $\phi$; see also Remark~\ref{rem:affine_structure}. Second, the compensator density \eqref{eq:hawkes_intensity_intro} is pointwise affine in the field. For Markov affine processes, conditional Fourier--Laplace transforms are exponentially affine in the state and the semimartingale characteristics have affine state dependence; see \cite{DuffiePanSingleton2000,DuffieFilipovicScachermayer2003}. In the present Volterra setting, the current value $V_t(x)$ alone need not be a sufficient state for the transform. Instead, the conditional transform is affine in the frozen forward input generated by the past, as in the affine Volterra setting of \cite{AbiJaberLarssonPulido2019}.

\subsection{Organization of the paper}
{\emergencystretch=3em
Section~\ref{sec:preliminaries} introduces the general setting and notation to be used throughout and establishes the key Lemma~\ref{lem:predictability}. Section~\ref{sec:existence} shows existence and uniqueness of the solution to equation~\eqref{eq:SVE_intro} for a large class of kernel functions $K$, which we term \emph{admissible} kernel functions. Additionally, Section~\ref{sec:affine_structure} elaborates on the affine structure of the model and relates it to existing models from the literature. Section~\ref{sec:spatial_transform} derives an affine transform formula for functionals of the form $\int_E f(y)V_t(y)m(dy)$ and Section~\ref{sec:spacetime_transform} extends this to the space-time setting of $\int_{0}^{t}\int_E f(s,y)V_s(y)m(dy)ds$. Section~\ref{sec:applications} considers several detailed examples of specifications of affine pure-jump Volterra fields and relates these to various Hawkes-type models previously considered in the literature. Finally, Section~\ref{sec:regularity} considers the spatial regularity of the field $V_t(x)$, such as conditions on $K$ that ensure Hölder continuity in mean for the map $x\mapsto V_t(x)$ and existence of a spatially Hölder continuous modification. Some lengthy proofs and supporting results are given in the Appendix.
\par} 

\section{Preliminaries}\label{sec:preliminaries}
In all of the following, we fix a complete probability space $(\Omega,\Fcal,\Prob)$ and a finite horizon $T>0$. We let $E$ be a Polish space with finite reference measure $m$ and Borel $\sigma$-algebra $\Ecal$. Our definitions of integer-valued random measures, predictable compensators, and stochastic integration is as in \cite[Chapter~II.1]{JacodShiryaev2003}. The proofs of Lemma~\ref{lem:distinct_times} and Lemma~\ref{lem:predictability} below are provided in Appendix~\ref{app:measurability}.

\begin{definition}\label{def:predictable}
For a right-continuous complete filtration $(\Fcal_t)$, let $\Pcal$ be the predictable $\sigma$-algebra on $\Omega\times[0,T]\times E$, generated by $A_0\times\{0\}\times B$ and $A\times(s,t]\times B$ with $A_0\in\Fcal_0$, $A\in\Fcal_s$, and $B\in\Ecal$. A random field is \emph{predictable} if it is $\Pcal$-measurable.
\end{definition}

Let $N$ be a Poisson random measure on $[0,T]\times E\times\Rp\times\Rp$ with deterministic intensity
\[
\Lambda(ds,dy,dz,du)=ds\,m(dy)\,\ell(dz)\,du,
\]
and let $(\Fcal_t^N)$ be its completed, right-continuous natural filtration,
\begin{equation}\label{eq:natural_filtration}
\Fcal_t^N=\bigcap_{r\in(t,T]}\Fcal_r^0\qquad(t<T),
\qquad \Fcal_T^N=\Fcal_T^0,
\end{equation}
where $\Fcal_t^0$ is generated by $N((0,s]\times B)$ for $s\leq t$ and augmented with all $\Prob$-null sets in $\Fcal$. The predictable compensator of $N$ is $\Lambda$, so every non-negative predictable integrand $H$ satisfies
\begin{equation}\label{eq:campbell}
\E\left[\int H\,dN\right]=\int\E[H] \,d\Lambda,
\end{equation}
with both sides in $[0,\infty]$. 

\begin{lemma}\label{lem:distinct_times}
  Let $N$ be a Poisson random measure on $[0,T]\times E\times\Rp\times\Rp$ with intensity $ds\otimes m(dy)\otimes\ell(dz)\otimes du$, where $m(E)<\infty$, $\ell(\lbrace0\rbrace)=0$, and $\ell((\varepsilon,\infty))<\infty$ for every $\varepsilon>0$. Then $N$ has no atom at time zero and, almost surely, its atoms have pairwise distinct time coordinates.
\end{lemma}
For a Borel kernel $K:(0,T]\times E\times E\to\Rp$ and a {$\Pcal\otimes\Bcal(\Rp)\otimes\Bcal(\Rp)$-measurable $W$ with values in $[0,1]$, where $\Pcal$ is taken relative to $(\Fcal_t^N)$}, set
\begin{equation}\label{eq:volterra_integral}
\Psi_t(x)=\int_{(0,t)\times E\times\Rp\times\Rp}K(t-s,x,y)z\,W(s,y,z,u)\,N(ds,dy,dz,du).
\end{equation}
The non-negative integral is defined pathwise in $[0,\infty]$. The open interval $(0,t)$ is our \emph{strict-past convention}, which permits $K$ to be singular at zero. Because the terminal index $t$ also enters the integrand, predictability of the resulting field is not automatic.

\begin{lemma}\label{lem:predictability}
Assume that $\ell((\varepsilon,\infty))<\infty$ for every $\varepsilon>0$. Then the field $\Psi$ in \eqref{eq:volterra_integral} is $\Fcal_t^N$-predictable, hence jointly measurable and adapted.
\end{lemma}

\section{Construction and affine structure}\label{sec:existence}
Throughout, the Poisson random measure $N$ has intensity $\Lambda$ given by
\begin{equation}\label{eq:N_intensity_measure}
\Lambda_N(ds,dy,dz,du)=ds\,m(dy)\,\ell(dz)\,du,
\end{equation}
where $\ell$ is supported on $(0,\infty)$ and
$\overline z:=\int_0^\infty z\,\ell(dz)<\infty$. Note in particular that this entails that $\ell((\varepsilon,\infty))<\infty$ for every $\varepsilon>0$, such that Lemma~\ref{lem:distinct_times} and Lemma~\ref{lem:predictability} apply. Let $K:(0,T]\times E\times E\to\Rp$ be a Borel kernel satisfying the following row-integrability condition.
\begin{definition}\label{def:admissible}
The kernel $K$ is \emph{admissible} if
\[
k(r):=\sup_{x\in E}\int_EK(r,x,y)\,m(dy)
\]
is Borel measurable and belongs to $L^1((0,T])$.
\end{definition}
We extend every admissible kernel to the diagonal by setting
\[
K(0,x,y):=0,\qquad x,y\in E.
\]
This convention leaves the strict-past stochastic integrals of the form \eqref{eq:volterra_integral} unchanged and simplifies deterministic formulas. It follows automatically from admissibility that

\[
\sup_{(t,x)\in[0,T]\times E}\int_{0}^{t}\int_E K(t-s,x,y)\,m(dy)\,ds \le \int_0^T k(r)\,dr<\infty.
\]
\begin{remark}\label{rem:coefficient_scope}
  We treat the scalar, non-negative-kernel case and use constants $\lambda_0,\lambda_1$ in the main statements to keep the transform equations transparent. Proposition~\ref{prop:inhomogeneous_coefficients} in Appendix~\ref{app:inhomogeneous_coefficients} records the extension to bounded non-negative spatial coefficients $\lambda_0(\cdot)$ and $\lambda_1(\cdot)$, including the modified contraction constant, compensator, Riccati--Volterra equations, and Laplace transforms. Vector- and operator-valued kernels require additional positivity structures and lie outside the present scope.
\end{remark}

\begin{example}\label{ex:admissible_kernels}
  The admissibility criteria include time-singular kernels. For example, let $G:E\times E\to\Rp$ be measurable with $\sup_{x\in E}\int_E G(x,y)\,m(dy)<\infty$. Then $K(r,x,y)=r^{\alpha-1}G(x,y)$ is admissible for every $\alpha\in(0,1)$ because $k(r)=r^{\alpha-1}\sup_x\int_EG(x,y)\,m(dy)$ and
  \[
  \int_{0}^Tk(r)\,dr = \frac{T^\alpha}{\alpha}\sup_{x\in E}\int_{E}G(x,y)\,m(dy)<\infty.
  \]
  It follows that kernels of the form $K(r,x,y)=r^{\alpha-1}p_r(x,y)$, where $p_r$ is a Markov transition density on $E$ with respect to $m$, are also admissible, since $\int_Ep_r(x,y)\,m(dy)=1$ for all $r> 0$.
\end{example}

For the remainder of this Section, fix constants $\lambda_0,\lambda_1\geq 0$, a measurable function $\phi:[0,T]\times E\to\Rp$, written $\phi_t(x)$, with $\lVert\phi\rVert_\infty := \sup_{(t,x)\in [0,T]\times E}\phi_t(x)<\infty$, and an admissible kernel function $K$.

\begin{definition}\label{def:strong_solution}
A \emph{PRM-strong predictable solution} is a non-negative $\Fcal_t^N$-predictable field with values in $[0,\infty]$, finite almost surely at each deterministic $(t,x)$, satisfying \eqref{eq:SVE_intro} almost surely at every deterministic $(t,x)$, on the prescribed Poisson space.
\end{definition}

For $\eta\geq 0$, we denote by $\mathcal{H}_\eta$ the space of all (equivalence classes of) $\Fcal_t^N$-predictable random fields $X$ with
\begin{equation}\label{eq:H_eta_norm}
\lVert X\rVert_\eta := \esssup_{(t,x)\in [0,T]\times E}e^{-\eta t}\,\E\left[ \lvert X_t(x)\rvert \right] < \infty,
\end{equation}
where the essential supremum is taken with respect to $dt\otimes m$ and two random fields are identified whenever they agree $\Prob\otimes dt\otimes m$-almost everywhere. The map $(t,x)\mapsto \E[\lvert X_t(x)\rvert]$ is measurable for every jointly measurable field $X$ by Tonelli's theorem, so that \eqref{eq:H_eta_norm} is well defined, and that $\lVert X\rVert_\eta = 0$ if and only if $X=0$ holds $\Prob\otimes dt\otimes m$-almost everywhere. Hence $\lVert\cdot\rVert_\eta$ is indeed a norm on $\mathcal{H}_\eta$. We let $\mathcal{H}_\eta^+\subseteq \mathcal{H}_\eta$ denote the cone of elements admitting a non-negative representative, i.e. $X_t(x)\geq 0$ almost surely and for $dt\otimes m$-almost all $(t,x)$. We further define
\begin{equation}\label{eq:J_def}
J_\eta(T) = \sup_{(t,x)\in [0,T]\times E}\int_{0}^{t}e^{-\eta (t-s)}\int_{E}K(t-s,x,y)\,m(dy)\,ds \leq \int_0^T e^{-\eta r}k(r)\,dr,
\end{equation}
and note that $J_\eta(T)\to0$ as $\eta\to\infty$ by dominated convergence.

For a non-negative predictable random field $X$, we define the random field $\Psi(X)$ pathwise by
\begin{equation}\label{eq:psi_operator}
\Psi (X)(t,x) = \phi_t(x) + \int_{(0,t)\times E\times \Rp\times \Rp} K(t-s,x,y)\,z\,\mathbf{1}_{\lbrace u\leq \lambda_0+\lambda_1 X_{s}(y) \rbrace}\,N(ds,dy,dz,du).
\end{equation}
Since $(\omega,s,y,z,u)\mapsto \mathbf{1}_{\lbrace u\leq \lambda_0+\lambda_1X_s(y)\rbrace}$ is $\Pcal\otimes\Bcal(\Rp)\otimes\Bcal(\Rp)$-measurable and $\phi$ is deterministic and measurable, Lemma~\ref{lem:predictability} shows that $\Psi(X)$ is a non-negative predictable random field. The following lemma contains the central $L^1$-estimate and in particular shows that $\Psi$ is well defined on equivalence classes.

\begin{lemma}\label{lem:L1_estimate}
  Let $X,\widetilde{X}$ be non-negative predictable random fields with $\lVert X\rVert_\eta<\infty$ and $\lVert\widetilde{X}\rVert_\eta<\infty$ for some $\eta\geq 0$. Then $\Psi(X)(t,x)$ and $\Psi(\widetilde{X})(t,x)$ as defined by \eqref{eq:psi_operator} are almost surely finite and satisfy
  \begin{equation}\label{eq:delta_bound}
  \E\left[\bigl\lvert\Psi(X)(t,x)-\Psi(\widetilde{X})(t,x)\bigr\rvert\right] \leq \overline{z}\lambda_1\int_{0}^{t}\int_{E}K(t-s,x,y)\, \E\left[ \lvert X_{s}(y)-\widetilde{X}_{s}(y)\rvert \right] m(dy)\,ds,
  \end{equation}
  for every $(t,x)\in [0,T]\times E$. If $X=\widetilde{X}$ holds $\Prob\otimes dt\otimes m$-almost everywhere, then $\Psi(X)(t,x)=\Psi(\widetilde{X})(t,x)$ almost surely, for every $(t,x)\in[0,T]\times E$.
\end{lemma}
\begin{proof}
  Since $\int_{0}^\infty\mathbf{1}_{\lbrace u\leq a\rbrace}\,du=a$ for $a\geq0$, equation \eqref{eq:campbell} yields
  \[
  \begin{aligned}
  \E\left[\Psi(X)(t,x)\right] &= \phi_t(x) + \overline{z}\int_{0}^{t}\int_E K(t-s,x,y)\left(\lambda_0+\lambda_1\E\left[X_s(y)\right]\right) m(dy)\,ds \\
  &\leq \lVert\phi\rVert_\infty + \overline{z}\left(\lambda_0+\lambda_1 e^{\eta T}\lVert X\rVert_\eta\right)\int_0^Tk(r)\,dr,
  \end{aligned}
  \]
  where we used that $\E[X_s(y)]\leq e^{\eta s}\lVert X\rVert_\eta\le e^{\eta T}\lVert X\rVert_\eta$ for $ds\otimes m$-a.e. $(s,y)$, which suffices since the exceptional set is not charged by $m(dy)\,ds$. Hence $\Psi(X)(t,x)<\infty$ almost surely, and similarly for $\widetilde{X}$, so that the difference $\Delta_t(x):=\Psi(X)(t,x)-\Psi(\widetilde{X})(t,x)$ is almost surely well defined. Using that $\lvert \mathbf{1}_{\lbrace u\leq a\rbrace}-\mathbf{1}_{\lbrace u\leq b\rbrace}\rvert = \mathbf{1}_{\lbrace a\wedge b< u\leq a\vee b\rbrace}$, we obtain the pathwise inequality
  \[
  \lvert \Delta_t(x)\rvert \leq \int_{(0,t)\times E \times \Rp\times \Rp}K(t-s,x,y)\,z\,\bigl\lvert \mathbf{1}_{\lbrace u\leq \lambda_0+\lambda_1X_{s}(y)\rbrace} - \mathbf{1}_{\lbrace u \leq \lambda_0 + \lambda_1\widetilde{X}_{s}(y)\rbrace}\bigr\rvert\, N(ds,dy,dz,du),
  \]
  whose integrand is again $\Pcal\otimes\Bcal(\Rp)\otimes\Bcal(\Rp)$-measurable. Taking expectations, applying the formula \eqref{eq:campbell} and using that $\int_{0}^\infty \lvert \mathbf{1}_{\lbrace u\leq a\rbrace}-\mathbf{1}_{\lbrace u\leq b\rbrace}\rvert\, du =\lvert a - b\rvert$ yields \eqref{eq:delta_bound}. Finally, if $X=\widetilde{X}$ holds $\Prob\otimes dt\otimes m$-a.e., then $\E[\lvert X_s(y)-\widetilde{X}_s(y)\rvert]=0$ for $ds\otimes m$-a.e. $(s,y)$, so the right-hand side of \eqref{eq:delta_bound} vanishes and $\Delta_t(x)=0$ almost surely, for every $(t,x)$.
\end{proof}

\begin{theorem}\label{thm:Psi_contraction}
  The operator $\Psi$ defined by \eqref{eq:psi_operator} maps $\mathcal{H}_\eta^+$ into itself for every $\eta\geq 0$ and satisfies
  \begin{equation}\label{eq:Psi_bound}
  \lVert \Psi(X)\rVert_\eta \leq \lVert \phi\rVert_{\infty}+\overline{z}\left(\lambda_0+\lambda_1\lVert X\rVert_\eta\right)J_\eta(T), \qquad X\in\mathcal{H}_\eta^+.
  \end{equation}
  Moreover, for every $\eta>0$ such that $\overline{z}\lambda_1 J_\eta(T)<1$, the operator $\Psi$ is a contraction on $\mathcal{H}_\eta^+$ and therefore admits a unique fixed point $Y\in\mathcal{H}_\eta^+$, i.e., $Y=\Psi(Y)$ in $\mathcal{H}_\eta$.
\end{theorem}
\begin{proof}
  Let $X\in\mathcal{H}_\eta^+$ with non-negative predictable representative, again denoted $X$. By Lemma~\ref{lem:predictability}, $\Psi(X)$ is a non-negative predictable random field, and by Lemma~\ref{lem:L1_estimate} its equivalence class does not depend on the choice of representative. As in the proof of Lemma~\ref{lem:L1_estimate}, equation \eqref{eq:campbell} yields that for every $(t,x)\in[0,T]\times E$,
  \[
  \begin{aligned}
  e^{-\eta t}\,\E\left[\Psi(X)(t,x)\right] &= e^{-\eta t}\phi_t(x) + \overline{z}\int_{0}^{t}\int_E e^{-\eta(t-s)}K(t-s,x,y)\,e^{-\eta s}\left(\lambda_0+\lambda_1\E\left[X_s(y)\right]\right) m(dy)\,ds \\
  &\leq \lVert\phi\rVert_\infty + \overline{z}\left(\lambda_0+\lambda_1\lVert X\rVert_\eta\right)J_\eta(T),
  \end{aligned}
  \]
  where we have used that $e^{-\eta s}\lambda_0\le\lambda_0$ and $e^{-\eta s}\E[X_s(y)]\leq \lVert X\rVert_\eta$ for $ds\otimes m$-a.e. $(s,y)$ (again the exceptional $ds\otimes m$-null set does not contribute to the integral). Taking the essential supremum over $(t,x)$ yields \eqref{eq:Psi_bound}, and we conclude that $\Psi(X)\in\mathcal{H}_\eta^+$. For the contraction property, let $X,\widetilde{X}\in\mathcal{H}_\eta^+$. Multiplying \eqref{eq:delta_bound} of Lemma~\ref{lem:L1_estimate} by $e^{-\eta t}$ gives
  \[
  \begin{aligned}
  e^{-\eta t}\,&\E\left[\bigl\lvert\Psi(X)(t,x)-\Psi(\widetilde{X})(t,x)\bigr\rvert\right] \\
  &\leq \overline{z}\lambda_1\int_{0}^{t}\int_{E}e^{-\eta(t-s)}K(t-s,x,y)\,e^{-\eta s}\,\E\left[ \lvert X_{s}(y)-\widetilde{X}_{s}(y)\rvert \right] m(dy)\,ds \\
  &\leq \overline{z}\lambda_1 J_\eta(T)\,\lVert X-\widetilde{X}\rVert_\eta,
  \end{aligned}
  \]
  where we have used the bound $e^{-\eta s}\E[\lvert X_s(y)-\widetilde{X}_s(y)\rvert]\leq \lVert X-\widetilde{X}\rVert_\eta$, which is valid for $ds\otimes m$-a.e. $(s,y)$. Hence
  \[
  \lVert \Psi (X)-\Psi(\widetilde{X})\rVert_\eta \leq \overline{z}\lambda_1 J_\eta (T)\,\lVert X-\widetilde{X}\rVert_\eta.
  \]
  Since $J_\eta(T)\to0$ as $\eta\to\infty$, we may choose $\eta>0$ with $\overline{z}\lambda_1J_\eta(T)<1$, and then $\Psi$ is a contraction on $\mathcal{H}_\eta^+$, which is a closed subset of the Banach space $\mathcal{H}_\eta$ by Lemma~\ref{lem:H_eta_Banach}. The Banach fixed point theorem yields a unique fixed point $Y\in\mathcal{H}_\eta^+$.
\end{proof}

The fixed point in $\mathcal H_\eta$ admits the following canonical pointwise representative.

\begin{lemma}\label{lem:pointwise_reconstruction}
Fix $\eta>0$ with $\overline z\lambda_1J_\eta(T)<1$, let $Y\in\mathcal H_\eta^+$ be the unique fixed point of $\Psi$, choose a non-negative predictable representative of $Y$, and define $V:=\Psi(Y)$ pathwise by \eqref{eq:psi_operator}. Then $V=Y$ holds $\Prob\otimes dt\otimes m$-almost everywhere and, for every deterministic $(t,x)\in[0,T]\times E$,
\[
V_t(x)=\Psi(V)(t,x)\qquad\text{almost surely}.
\]
Consequently, $V$ is a PRM-strong predictable solution in the sense of Definition~\ref{def:strong_solution}.
\end{lemma}
\begin{proof}
Since $Y=\Psi(Y)$ in $\mathcal H_\eta$ and $V=\Psi(Y)$ pathwise, $V=Y$ holds $\Prob\otimes dt\otimes m$-almost everywhere. Lemma~\ref{lem:L1_estimate} shows that $\Psi$ is insensitive, at each deterministic $(t,x)$, to a change of its input on a $\Prob\otimes dt\otimes m$-null set. Hence $\Psi(V)(t,x)=\Psi(Y)(t,x)=V_t(x)$ almost surely for every deterministic $(t,x)$.
\end{proof}

\begin{theorem}\label{thm:strong_existence}
  There exists a PRM-strong predictable solution $V=(V_t(x))_{(t,x)\in[0,T]\times E}$ in the sense of Definition~\ref{def:strong_solution}, with
  \begin{equation}\label{eq:V_sup_moment_bound}
  \sup_{(t,x)\in [0,T]\times E}\E\left[ V_t(x) \right] < \infty,
  \end{equation}
  such that for every $(t,x)\in[0,T]\times E$ the stochastic Volterra equation
  \begin{equation}\label{eq:strong_solution}
  V_t(x) = \phi_t(x) + \int_{(0,t)\times E\times \Rp\times \Rp} K(t-s, x, y) \, z \, \mathbf{1}_{\lbrace {u \le \lambda_0 + \lambda_1 V_{s}(y)}\rbrace}\, N(ds, dy, dz, du),
  \end{equation}
  holds almost surely. Moreover, $V$ is unique in the following senses. If $\widetilde{V}$ is another non-negative predictable random field satisfying \eqref{eq:V_sup_moment_bound} and solving \eqref{eq:strong_solution} almost surely for every $(t,x)\in[0,T]\times E$, then $V=\widetilde{V}$ holds $\Prob\otimes dt \otimes m$-almost everywhere. For every fixed $(t,x)\in[0,T]\times E$, it holds that $\Prob(V_t(x)=\widetilde{V}_t(x))=1$. Therefore, for every fixed $t\in[0,T]$, $V_t(\cdot)=\widetilde{V}_t(\cdot)$ holds $m$-almost everywhere, almost surely.
\end{theorem}

The canonical convolution $V=\Psi(Y)$ is understood with values in $[0,\infty]$ and is finite almost surely at each deterministic $(t,x)$ and outside a $\Prob\otimes dt\otimes m$-null set. Simultaneous identities refer to this representative and do not claim simultaneous finiteness over all $(t,x)$. Replacing exceptional infinite values arbitrarily need not preserve such identities and for any non-negative integrals and products we use the convention $0\cdot\infty=0$.

\begin{remark}\label{rem:random_phi}
  The driver $\phi$ may itself be stochastic. Let $\Fcal_0\subseteq\Fcal$ be a sub-$\sigma$-algebra and let $(\widetilde{\Fcal}_t)$ be the usual augmentation of $(\Fcal_0\vee\Fcal_t^N)$. Assume that $\phi:\Omega\times[0,T]\times E\to\Rp$ is $\Fcal_0\otimes\Bcal([0,T])\otimes\Ecal$-measurable and satisfies $\sup_{(t,x)}\E[\phi_t(x)]<\infty$. Assume also that $\Lambda_N$ remains the predictable compensator of $N$ in $(\widetilde{\Fcal}_t)_{t\geq 0}$. Under these assumptions, Theorems~\ref{thm:Psi_contraction} and~\ref{thm:strong_existence} remain valid with predictability relative to $\widetilde{\Fcal}_t$ and with $\lVert\phi\rVert_\infty$ replaced by $\sup_{(t,x)}\E[\phi_t(x)]$ in the estimates. The pointwise first-moment bound on $\phi$ preserves \eqref{eq:V_sup_moment_bound} and the almost-sure solution identity for every $(t,x)$. Independence of $\Fcal_0$ and $\Fcal^N$ is a sufficient condition for the compensator assumption and therefore for equation~\eqref{eq:campbell} in the enlarged filtration.
\end{remark}

\subsection{Affine jump measure}\label{sec:affine_structure}
Define the accepted-jump measure $\mu$ on $(0,T]\times E\times\Rp$ by
\begin{equation}\label{eq:thinned_measure}
\mu(A) := \int_{(0,T]\times E\times\Rp\times\Rp}\mathbf{1}_{A}(s,y,z)\,\mathbf{1}_{\lbrace u\leq \lambda_0+\lambda_1V_{s}(y)\rbrace}(u)\,N(ds,dy,dz,du).
\end{equation}
Thus $\mu$ is an integer-valued random measure. If $\ell((0,\infty))<\infty$, it is locally finite. If $\ell((0,\infty))=\infty$, it has infinite activity whenever the integrated acceptance intensity is positive. The stochastic convolution remains finite because $\overline z<\infty$.

\begin{proposition}\label{prop:affine_compensator}
  The predictable compensator (dual predictable projection) of $\mu$ is given by
  \begin{equation}\label{eq:affine_compensator}
  \nu(ds,dy,dz) = \left(\lambda_0+\lambda_1V_{s}(y)\right)ds\,m(dy)\,\ell(dz).
  \end{equation}
{\emergencystretch=3em
  Thus $\Lambda_s(y)=\lambda_0+\lambda_1V_s(y)$ is the predictable density multiplier of the accepted activity relative to $ds\,m(dy)\,\ell(dz)$. It is a finite total jump-rate density only when the relevant mark restriction has finite $\ell$-mass.
\par}
\end{proposition}
\begin{proof}
  Since $V$ is predictable and non-negative, $\nu$ is predictable. Now localize the mark space by
  \[
  C_n:=E\times(1/n,n],\qquad n\in\N.
  \]
  These sets increase to $E\times(0,\infty)$, and from \eqref{eq:campbell} and \eqref{eq:V_sup_moment_bound}, we find that
  \[
  \E[\mu((0,T]\times C_n)]
  =\E[\nu((0,T]\times C_n)]
  \leq(\lambda_0+\lambda_1M)T\,m(E)\,\ell((1/n,n])<\infty,
  \]
  where $M:=\sup_{(s,y)}\E[V_s(y)]$.

  If $A\subseteq C_n$ is Borel, then the process $t\mapsto\mu((0,t]\times A)$ is adapted and has finite total count almost surely by the preceding expectation bound. As the cumulative mass of a finite point measure on $(0,T]$, it is c\`adl\`ag and hence optional (see \cite[Chapter~II.1.a]{JacodShiryaev2003}). Thus every localized measure $\mathbf1_{C_n}\mu$ is optional and the increasing exhaustion also establishes that the full measure $\mu$ is optional and $\sigma$-finite on $E$, as required in the following. 

  Now let $H$ be a non-negative $\Pcal\otimes\Bcal(\Rp)$-measurable integrand supported on $(0,T]\times C_n$. Equation~\eqref{eq:campbell}, Tonelli's theorem, and the identity $\int_0^\infty\mathbf1_{\{u\leq a\}}\,du=a$ yield
  \[
  \begin{aligned}
  \E\left[\int_{(0,T]\times E\times\Rp} H\,d\mu\right] &= \int_{(0,T]\times E\times\Rp}\E\left[H(s,y,z)\left(\lambda_0+\lambda_1V_s(y)\right)\right]ds\,m(dy)\,\ell(dz) \\
  &= \E\left[\int_{(0,T]\times E\times\Rp} H\,d\nu\right].
  \end{aligned}
  \]
  Monotone convergence extends the identity to arbitrary non-negative predictable $H$ and the finite expectations on the deterministic exhaustion $(C_n)$ give predictable $\sigma$-integrability for both measures. The compensator characterization and uniqueness in \cite[Theorem~II.1.8]{JacodShiryaev2003} now identify $\nu$ as the predictable compensator of $\mu$.
\end{proof}

\begin{corollary}\label{cor:activity_dichotomy}
  Let $I\subseteq(0,T]$ and $B\in\Ecal$ be Borel sets and set
  \[
  \xi := \int_I\int_B\left(\lambda_0+\lambda_1V_s(y)\right)m(dy)\,ds,
  \]
  which is almost surely finite by \eqref{eq:V_sup_moment_bound}. Then $\mu(I\times B\times\Rp)=0$ almost surely on $\lbrace\xi=0\rbrace$. If $\ell((0,\infty))<\infty$, then $\mu(I\times B\times\Rp)<\infty$ almost surely and if $\ell((0,\infty))=\infty$, then $\mu(I\times B\times\Rp)=\infty$ almost surely on $\lbrace\xi>0\rbrace$.
\end{corollary}

\begin{remark}\label{rem:hawkes_random_measure}
  Suppose that $\lambda_1>0$. By the definition \eqref{eq:thinned_measure} of $\mu$, the stochastic Volterra equation \eqref{eq:strong_solution} states that, for every $(t,y)\in[0,T]\times E$, almost surely
  \[
  \lambda_0+\lambda_1V_t(y)
  =\lambda_0+\lambda_1\phi_t(y)
  +\int_{(0,t)\times E\times\Rp}\lambda_1K(t-s,y,w)\,z\,\mu(ds,dw,dz).
  \]
  Together with Proposition~\ref{prop:affine_compensator}, it identifies $\mu$ as a \emph{marked Hawkes random measure} on $E$. Its compensator is $\Lambda_s(y)\,ds\,m(dy)\,\ell(dz)$, its baseline is $\lambda_0+\lambda_1\phi$, and its excitation kernel is $\lambda_1K(t-s,y,w)z$. When $\ell$ has infinite mass, $\Lambda$ describes activity relative to the mark measure. The objects studied below are the propagated field $V$, equivalently, the affine rescaling of $\Lambda$, and its transform theory.
\end{remark}

\begin{proposition}\label{prop:pathwise_identity}
  Almost surely, the identity
  \begin{equation}\label{eq:pathwise_identity}
  V_t(x) = \phi_t(x) + \int_{(0,t)\times E\times\Rp}K(t-s,x,y)\,z\,\mu(ds,dy,dz),
  \end{equation}
  holds simultaneously for all $(t,x)\in[0,T]\times E$ for the canonical representative, with the equality understood in $[0,\infty]$.
\end{proposition}
\begin{proof}
  Let $Y$ denote the fixed-point representative from the proof of Theorem~\ref{thm:strong_existence}, so that $V=\Psi(Y)$ pathwise; that is, for every $(\omega,r,y)$,
  \begin{equation}\label{eq:V_pathwise_Y}
  V_r(y) = \phi_r(y) + \int_{(0,r)\times E\times\Rp\times\Rp}K(r-u,y,w)\,z\,\mathbf{1}_{\lbrace v\leq\lambda_0+\lambda_1Y_u(w)\rbrace}\,N(du,dw,dz,dv).
  \end{equation}
  Consider the predictable set
  \[
  D:=\left\{(\omega,u,w,z,v):
  \mathbf{1}_{\{v\leq\lambda_0+\lambda_1V_u(w)\}}
  \neq
  \mathbf{1}_{\{v\leq\lambda_0+\lambda_1Y_u(w)\}}
  \right\}.
  \]
  Infinite activity of $\ell$ prevents us from integrating out the $z$-coordinate directly, so for $n\in\N$, set
  \[
  D_n:=D\cap\left\{n^{-1}<z\leq n,\ v\leq n\right\}.
  \]
  Equation~\ref{eq:campbell} and the identity
  $\int_0^n|\mathbf{1}_{\{v\leq a\}}-\mathbf{1}_{\{v\leq b\}}|\,dv\leq |a-b|$
  yield
  \[
  \E[N(D_n)]
  \leq \lambda_1\ell((n^{-1},n])
  \E\left[\int_0^T\int_E|V_u(w)-Y_u(w)|\,m(dw)\,du\right]=0,
  \]
  because $V=Y$ holds $\Prob\otimes dt\otimes m$-almost everywhere. Thus $\mathbb{P}(N(D_n)=0)=1$ for every $n$. Since $D=\bigcup_{n\geq1}D_n$ up to the set $\{z=0\}$, which is not charged by $N$, we have $N(D)=0$ almost surely. On this single event, the two thinning indicators agree $N$-almost everywhere, so the accepted-jump measures built from $V$ and from $Y$ coincide as measures on $(0,T]\times E\times\Rp$. Hence, for every non-negative Borel function $g$ on $(0,T]\times E\times\Rp$, it holds that
  \[
  \int_{(0,T]\times E\times\Rp}g\,d\mu = \int_{(0,T]\times E\times\Rp\times\Rp}g(u,w,z)\,\mathbf{1}_{\lbrace v\leq\lambda_0+\lambda_1Y_u(w)\rbrace}\,N(du,dw,dz,dv).
  \]
  Applying this with $g(u,w,z)=\mathbf{1}_{\lbrace u<r\rbrace}K(r-u,y,w)\,z$ for each $(r,y)$ and comparing with \eqref{eq:V_pathwise_Y} yields \eqref{eq:pathwise_identity} simultaneously for all $(r,y)$.
\end{proof}

For a predictable integrand $H$ for which the $\mu$- and $\nu$-integrals of $|H|$ over $(0,t]\times E\times\Rp$ are almost surely finite, we use the standard compensated-integral notation
\[
H\star(\mu-\nu)_t
:=\int_{(0,t]\times E\times\Rp}H(s,y,z)\,\mu(ds,dy,dz)
-\int_{(0,t]\times E\times\Rp}H(s,y,z)\,\nu(ds,dy,dz).
\]

\begin{corollary}\label{cor:drift_martingale}
  For every $(t,x)\in[0,T]\times E$ it holds almost surely that
  \begin{equation}\label{eq:drift_martingale_form}
  \begin{aligned}
  V_t(x)={}&\phi_t(x)
  +\overline{z}\int_0^t\int_E K(t-s,x,y)
  \left(\lambda_0+\lambda_1V_s(y)\right)m(dy)\,ds\\
  &+H_{t,x}\star(\mu-\nu)_t,
  \end{aligned}
  \end{equation}
  where $H_{t,x}(s,y,z):=\mathbf 1_{\{s<t\}}K(t-s,x,y)z$. The two finite-variation integrals defining $H_{t,x}\star(\mu-\nu)_t$ are almost surely finite. For fixed $(t,x)$ the compensated integral has mean zero, but the map $t\mapsto H_{t,x}\star(\mu-\nu)_t$ need not be a martingale because the integrand itself depends on the terminal time $t$.
\end{corollary}
\begin{proof}
{\emergencystretch=3em
  By the definition \eqref{eq:thinned_measure} of $\mu$, the stochastic Volterra equation \eqref{eq:strong_solution} reads $V_t(x) = \phi_t(x) + \int_{(0,t)\times E\times\Rp}K(t-s,x,y)\,z\,\mu(ds,dy,dz)$ almost surely. Applying the compensator identity from the proof of Proposition~\ref{prop:affine_compensator} to the deterministic integrand $(s,y,z)\mapsto \mathbf{1}_{(0,t)}(s)K(t-s,x,y)\,z$ gives
\par}
  \[
  \E\left[\int_{(0,t)\times E\times\Rp}K(t-s,x,y)\,z\,\mu(ds,dy,dz)\right] = \overline{z}\int_0^t\int_E K(t-s,x,y)\left(\lambda_0+\lambda_1\E\left[V_s(y)\right]\right)m(dy)\,ds,
  \]
  which is finite by \eqref{eq:V_sup_moment_bound} and the admissibility of $K$. Hence the $\mu$- and $\nu$-integrals of $\mathbf{1}_{(0,t)}(s)K(t-s,x,y)\,z$ are both almost surely finite, and subtracting the pathwise identity 
  \[
  \int_{(0,t)\times E\times\Rp}K(t-s,x,y)\,z\,\nu(ds,dy,dz)=\overline{z}\int_0^t\int_EK(t-s,x,y)(\lambda_0+\lambda_1V_s(y))\,m(dy)\,ds,
  \]
  from \eqref{eq:strong_solution} yields \eqref{eq:drift_martingale_form}.
\end{proof}

\begin{remark}\label{rem:martingale_measure}
  If, in addition to the already imposed regularity, it holds that $\overline{z}_2:=\int_0^\infty z^2\,\ell(dz)<\infty$, then
  \[
  \E\left[\int_0^T\int_E\int_0^\infty z^2\,\nu(ds,dy,dz)\right]
  \leq \overline{z}_2(\lambda_0+\lambda_1M)T\,m(E)<\infty,
  \]
  where $M:=\sup_{(s,y)\in[0,T]\times E}\E[V_s(y)]<\infty$.
  For measurable $A\in\Ecal$, define the process
  \[
  \mathcal{M}_t(A):=\bigl(z\mathbf 1_A(y)\bigr)\star(\mu-\nu)_t.
  \]
  Then $t\mapsto\mathcal{M}_t(A)$ is a square-integrable martingale and compensated jumps on disjoint spatial sets have zero predictable covariation. Thus $\mathcal{M}$ is an orthogonal martingale measure in the sense of \cite{Walsh1986}, with predictable covariance measure
  \[
  Q_{\mathcal M}(ds,dy)
  =\overline{z}_2\left(\lambda_0+\lambda_1V_s(y)\right)ds\,m(dy).
  \]
  Equation~\eqref{eq:drift_martingale_form} can therefore be read as a stochastic Volterra equation driven by $\mathcal{M}$. In particular, the martingale-measure representation gives an It\^o isometry for square-integrable integrands. We refer to \cite{Walsh1986,ChongKluppelberg2015} for the associated integration theory.
\end{remark}

\begin{example}[Ambit fields]\label{ex:ambit_field}
  If $\lambda_1=0$, then $\nu(ds,dy,dz) = \lambda_0\,ds\,m(dy)\,\ell(dz)$ is deterministic, and $\mu$ is itself a Poisson random measure. In this case \eqref{eq:strong_solution} is a convolution of ambit or shot-noise type driven by a non-negative finite-variation L\'evy basis, cf. \cite{Ambit}. This does not imply finite-variation temporal paths of $V$, since neither $K$ nor $\phi$ is assumed to have bounded variation. It is a subclass of L\'evy-driven ambit fields, since the present framework excludes Gaussian components, signed jumps, and infinite-variation L\'evy bases. More precisely, if $K(r,x,y)$ is the transition density of a suitable strongly continuous Markov semigroup $(P_r)_{r\geq0}$ and $\phi_t=P_t\phi_0$, then the equation has the standard stochastic-convolution form of a mild evolution equation with Poisson-type additive noise, subject to the usual integrability conditions; see \cite{PeszatZabczyk2007} for the general evolution-equation setting.
\end{example}

\begin{example}[Affine Volterra jump processes]\label{ex:affine_pure_jump}
  Let $E=\lbrace e\rbrace$ be a singleton with $m=\delta_e$ (recall that $m$ is not required to be atomless) and $K(r,e,e)=K(r)$ with $K\in L^1([0,T])$. Then \eqref{eq:strong_solution} reduces to the scalar equation
  \[
  V_t = \phi_t + \int_{(0,t)\times\Rp\times\Rp}K(t-s)\,z\,\mathbf{1}_{\lbrace u\leq \lambda_0+\lambda_1V_s\rbrace}\,N(ds,dz,du),
  \]
  an affine Volterra jump process with affine activity density in the spirit of \cite{BondiLivieriPulido2024}, and Proposition~\ref{prop:affine_compensator} recovers the compensator $(\lambda_0+\lambda_1V_s)\,ds\,\ell(dz)$. For $K\equiv1$ and $\phi\equiv v_0\geq0$, the closed-interval version -- obtained by integrating over $(0,t]$ -- is a classical Markovian affine pure-jump process and agrees with $V_t$ almost surely at each deterministic time. Taking $E$ finite with counting measure yields multivariate Hawkes-type systems with Volterra kernels; see Section~\ref{subsec:hawkes} for an example of such a construction.
\end{example}

\begin{example}[$V$ as an $L^2$-valued process]\label{ex:L2_process}
  Suppose that $\sup_{(t,x)\in [0,T]\times E}\mathbb{E}\left[V_t(x)^p\right]<\infty$ for some $p\geq 2$ (for example from the assumptions of Proposition~\ref{prop:p_moments}). Since $m(E)<\infty$, it follows that
  \[
  \sup_{t\in[0,T]}\E\left[\lVert V_t(\cdot)\rVert_{L^2(E,m)}^2\right]
  \leq m(E)\times\sup_{(t,x)\in[0,T]\times E}\E[V_t(x)^2]<\infty,
  \]
  and thus that $\mathbb{E}\left[\int_0^T\lVert V_t(\cdot)\rVert_{L^2(E,m)}^2dt\right]<\infty$. Joint measurability of $V$ and separability of $L^2(E,m)$ imply that $(\omega,t)\mapsto[V_t(\cdot;\omega)]_{L^2(E,m)}$ has a strongly/Bochner measurable $L^2(E,m)$-valued representative outside a $\Prob\otimes dt$-null set. We define this representative to be zero on that set. For every fixed $t$, the uniform moment bound also shows that $V_t(\cdot)\in L^2(E,m)$ almost surely. This connects the field construction to Hilbert space-valued L\'evy models such as \cite{BenthKruhner2023,BenthSgarra2024}. The assertion is only an $L^2$-valued measurability statement up to product-null sets and gives neither temporal nor spatial path regularity, and in particular point evaluation $\delta_xV_t(\cdot)=V_t(x)$ need not be continuous. We refer to Section~\ref{sec:regularity} for conditions that ensure some spatial regularity.
\end{example}

We next derive a deterministic Volterra equation for the first moment of $V$. For this, define the space-time convolution powers of the kernel $K$ by $K^{\star 1}:=K$ and
\begin{equation}\label{eq:convolution_powers}
K^{\star (n+1)}(r,x,y) := \int_0^r\int_E K(r-s,x,w)\,K^{\star n}(s,w,y)\,m(dw)\,ds, \qquad n\in\N.
\end{equation}
Similarly, we write $k^{\star n}$ for the $n$-fold convolution power of the scalar function $k$ on $[0,T]$. An induction based on Tonelli's theorem shows that, for all $(r,x)\in[0,T]\times E$ and every $\eta\geq 0$,
\begin{equation}\label{eq:convolution_estimates}
\int_E K^{\star n}(r,x,y)\,m(dy)\leq k^{\star n}(r) \qquad\text{and}\qquad \int_0^Te^{-\eta r}k^{\star n}(r)\,dr \leq \left(\int_0^Te^{-\eta r}k(r)\,dr\right)^{n}.
\end{equation}

\begin{proposition}\label{prop:first_moment}
  The function $f(t,x):=\E\left[V_t(x)\right]$ is measurable and bounded on $[0,T]\times E$, and it is the unique bounded measurable solution of the linear Volterra equation
  \begin{equation}\label{eq:first_moment_equation}
  f(t,x) = \phi_t(x) + \overline{z}\int_0^t\int_E K(t-s,x,y)\left(\lambda_0+\lambda_1 f(s,y)\right)m(dy)\,ds, \qquad (t,x)\in [0,T]\times E.
  \end{equation}
{\emergencystretch=3em
  Moreover, the resolvent series $R:=\sum_{n=1}^{\infty}(\overline{z}\lambda_1)^nK^{\star n}$ converges pointwise in $[0,\infty]$ and satisfies $\int_E R(r,x,y)\,m(dy)\leq \rho(r)$ for all $(r,x)$, where the function $\rho:=\sum_{n=1}^\infty (\overline{z}\lambda_1)^nk^{\star n}$ satisfies $\int_0^Te^{-\eta r}\rho(r)\,dr<\infty$ whenever $\overline{z}\lambda_1\int_0^Te^{-\eta r}k(r)\,dr<1$. The function $f$ admits the representation
\par}
  \begin{equation}\label{eq:first_moment_resolvent}
  f(t,x) = g(t,x) + \int_0^t\int_E R(t-s,x,y)\,g(s,y)\,m(dy)\,ds,
  \end{equation}
  where $g(t,x):=\phi_t(x)+\overline{z}\lambda_0\int_0^t\int_EK(t-s,x,y)\,m(dy)\,ds$.
\end{proposition}
\begin{proof}
  Measurability of $f$ follows from the joint measurability of $V$ and Tonelli's theorem. Boundedness follows from \eqref{eq:V_sup_moment_bound} and \eqref{eq:first_moment_equation} follows by taking expectations in \eqref{eq:strong_solution} via equation~\eqref{eq:campbell}, exactly as in the proof of Theorem~\ref{thm:strong_existence}. For uniqueness, let $\widetilde{f}$ be another bounded measurable solution and set $h:=\lvert f-\widetilde{f}\rvert$. Subtracting the two equations gives the pointwise inequality 
  \[
  h(t,x)\leq \overline{z}\lambda_1\int_0^t\int_EK(t-s,x,y)\,h(s,y)\,m(dy)\,ds,
  \]
  for every $(t,x)$. It then follows that the supremum is bounded 
  \[
  \lVert h\rVert_{\eta,\infty}:=\sup_{(t,x)\in[0,T]\times E}e^{-\eta t}h(t,x)\leq \overline{z}\lambda_1 J_\eta(T)\,\lVert h\rVert_{\eta,\infty} < \infty,
  \]
  with $J_\eta(T)$ as in \eqref{eq:J_def}. Choosing $\eta$ with $\overline{z}\lambda_1J_\eta(T)<1$ yields $h\equiv 0$. For the resolvent representation, note first that all terms of $R$ are non-negative, so the pointwise convergence and the bound by $\rho$ are immediate from the estimates preceding the proposition, and 
  \[
  \int_0^Te^{-\eta r}\rho(r)\,dr\leq \sum_{n\geq1}\bigl(\overline{z}\lambda_1\int_0^Te^{-\eta r}k(r)\,dr\bigr)^n<\infty
  \]
  for $\eta$ as stated. Iterating \eqref{eq:first_moment_equation} $N$ times and an application of Tonelli's theorem to reorder the convolutions yields that
  \[
  f = g + \sum_{n=1}^{N}(\overline{z}\lambda_1)^n K^{\star n}\star g + (\overline{z}\lambda_1)^{N+1}K^{\star(N+1)}\star f,
  \]
  where $(F\star h)(t,x):=\int_0^t\int_EF(t-s,x,y)\,h(s,y)\,m(dy)\,ds$. The remainder satisfies
  \[
  e^{-\eta t}\,(\overline{z}\lambda_1)^{N+1}\bigl(K^{\star(N+1)}\star f\bigr)(t,x) \leq \lVert f\rVert_\infty \Bigl(\overline{z}\lambda_1\int_0^Te^{-\eta r}k(r)\,dr\Bigr)^{N+1} \xrightarrow[N\to\infty]{} 0
  \]
  for every fixed $(t,x)$, while the partial sums increase to $R\star g$ by monotone convergence. This proves \eqref{eq:first_moment_resolvent}.
\end{proof}

\section{Fixed-time affine transforms}\label{sec:spatial_transform}
In this section, we derive the Laplace transform of $\langle f,V_t\rangle$ for a bounded and positive Borel function $f\in B_b(E)_+$. The pairing $\langle f,h\rangle$ for measurable $h:E\to[0,\infty]$ is defined by
\[
\langle f,h\rangle := \int_E f(x)\,h(x)\,m(dx).
\]
Since $V$ is jointly measurable and satisfies \eqref{eq:V_sup_moment_bound}, the pairing $\langle f,V_t\rangle$ is for every $t\in[0,T]$ a well-defined random variable with $\E[\langle f,V_t\rangle]\leq \lVert f\rVert_\infty\,m(E)\,M<\infty$, where here and in the remainder of the paper $M:=\sup_{(t,x)\in[0,T]\times E}\E[V_t(x)]$. Define the \emph{dual kernel} $K_f$ by
\begin{equation}\label{eq:dual_kernel}
K_f(r,y) := \int_E f(x)\,K(r,x,y)\,m(dx), \qquad (r,y)\in(0,T]\times E,
\end{equation}
which is Borel measurable, and define the \emph{L\'evy exponent}
\begin{equation}\label{eq:levy_exponent}
F(q) := \int_0^\infty\left(1-e^{-qz}\right)\ell(dz), \qquad q\in[0,\infty],
\end{equation}
where we set $F(\infty):=\ell((0,\infty))\in[0,\infty]$. By monotone convergence and the elementary bounds $1-e^{-a}\leq a$ and $\lvert e^{-qz}-e^{-q'z}\rvert\leq z\lvert q-q'\rvert$, the function $F$ is non-decreasing and continuous on $[0,\infty]$ with
\begin{equation}\label{eq:F_properties}
F(0)=0,\qquad F(q)\leq \overline{z}\,q,\qquad \lvert F(q)-F(q')\rvert\leq \overline{z}\,\lvert q-q'\rvert,\qquad q,q'\in[0,\infty).
\end{equation}

Admissibility controls the rows of $K(r,\cdot,\cdot)$, while the transform problem acts in the dual spatial direction and requires the corresponding column bound.

\begin{assumption}\label{ass:K_transform}
  The \emph{dual majorant} $\widetilde{k}(r) := \sup_{y\in E}\int_E K(r,x,y)\,m(dx)$ is Borel measurable on $(0,T]$ and satisfies $\int_0^T\widetilde{k}(r)\,dr<\infty$.
\end{assumption}

\begin{remark}\label{rem:K_transform_scope}
If $K(r,x,y)=K(r,y,x)$, then $\widetilde k=k$. The dual condition is also automatic when $E$ is finite with counting measure. In general the row and column assumptions are distinct.
\end{remark}

Define the \emph{dual resolvent}
\begin{equation}\label{eq:dual_resolvent}
\widetilde{\rho} := \widetilde{k}+\sum_{n=2}^{\infty}(\lambda_1\overline{z})^{n-1}\,\widetilde{k}^{\star n},
\end{equation}
where $\widetilde{k}^{\star n}$ is the $n$-fold scalar convolution. As in Proposition~\ref{prop:first_moment}, it holds that
\[
\int_0^Te^{-\eta r}\widetilde{k}^{\star n}(r)\,dr\leq (\widetilde{J}_\eta)^{n}, \qquad \widetilde{J}_\eta=\int_0^Te^{-\eta r}\widetilde{k}(r)\,dr,
\]
and since $\widetilde{J}_\eta\to0$ as $\eta\to\infty$, choosing $\eta$ with $\lambda_1\overline{z}\widetilde{J}_\eta<1$ yields
\begin{equation}\label{eq:dual_resolvent_L1}
\int_0^T\widetilde{\rho}(r)\,dr \leq e^{\eta T}\int_0^Te^{-\eta r}\widetilde{\rho}(r)\,dr \leq \frac{e^{\eta T}\,\widetilde{J}_\eta}{1-\lambda_1\overline{z}\widetilde{J}_\eta} < \infty.
\end{equation}
Hence $\widetilde\rho\in L^1((0,T])$ and is finite almost everywhere.

\begin{lemma}\label{lem:riccati_volterra}
  Let $f\in B_b(E)_+$. Then there exists a Borel measurable function $\psi_f:(0,T]\times E\to[0,\infty]$ satisfying the Riccati--Volterra equation
  \begin{equation}\label{eq:riccati_volterra}
  \psi_f(r,y) = K_f(r,y) + \lambda_1\int_0^r\int_E K(r-s,y',y)\,F\bigl(\psi_f(s,y')\bigr)\,m(dy')\,ds, \qquad (r,y)\in(0,T]\times E,
  \end{equation}
  together with the bound
  \begin{equation}\label{eq:riccati_bound}
  \psi_f(r,y) \leq \lVert f\rVert_\infty\,\widetilde{\rho}(r), \qquad (r,y)\in(0,T]\times E.
  \end{equation}
  It holds that $\psi_f(r,y)<\infty$ for every $y$ and every $r$ outside a deterministic Lebesgue-null set. Additionally, the function $\psi_f$ is minimal among all non-negative Borel solutions of \eqref{eq:riccati_volterra}, and it is the unique solution in the following dominated class: if $\widetilde{\psi}\geq0$ is a Borel solution of \eqref{eq:riccati_volterra} such that $\widetilde{\psi}(r,y)\leq\Theta(r)$ for all $(r,y)$ and some $\Theta\in L^1((0,T])$, then $\widetilde{\psi}=\psi_f$ everywhere on $(0,T]\times E$. If $\lambda_1=0$, then $\psi_f=K_f$.
\end{lemma}
\begin{proof}
  Denote by $\mathcal{T}$ the operator mapping a Borel function $\psi:(0,T]\times E\to[0,\infty]$ to the right-hand side of \eqref{eq:riccati_volterra}. By Tonelli's theorem and the continuity of $F$ on $[0,\infty]$, the function $\mathcal{T}\psi$ is again Borel measurable with values in $[0,\infty]$ and $\mathcal{T}$ is monotone in the sense that $\psi\leq\widetilde{\psi}$ pointwise implies $\mathcal{T}\psi\leq\mathcal{T}\widetilde{\psi}$, since $F$ is non-decreasing. Define the iteration 
  \[
  \psi^{(0)}:=K_f, \qquad \psi^{(m+1)}:=\mathcal{T}\psi^{(m)}.
  \]
  Since $\mathcal{T}\psi\geq K_f$ for every $\psi\geq0$, we have $\psi^{(1)}\geq\psi^{(0)}$, and by monotonicity of $\mathcal{T}$ the sequence $(\psi^{(m)})_{m\geq0}$ is non-decreasing. Its pointwise limit $\psi_f=\sup_m\psi^{(m)}$ is Borel measurable, and since $F$ is continuous and non-decreasing on $[0,\infty]$, monotone convergence applied to both sides of $\psi^{(m+1)}=\mathcal{T}\psi^{(m)}$ yields that $\psi_f=\mathcal{T}\psi_f$, which is exactly \eqref{eq:riccati_volterra}.

  For \eqref{eq:riccati_bound}, $K_f(r,y)\leq\lVert f\rVert_\infty\,\widetilde{k}(r)$ by the definition of $\widetilde{k}$. Using $F(q)\leq\overline{z}q$ from \eqref{eq:F_properties} and $\int_EK(r-s,y',y)\,m(dy')\leq\widetilde{k}(r-s)$, an induction over $m$ shows that
  \[
  \psi^{(m)}(r,y) \leq \lVert f\rVert_\infty\left(\widetilde{k}(r)+\sum_{n=2}^{m+1}(\lambda_1\overline{z})^{n-1}\,\widetilde{k}^{\star n}(r)\right), \qquad m\geq0,
  \]
  and letting $m\to\infty$ gives \eqref{eq:riccati_bound}. Minimality follows by induction: if $\widetilde{\psi}\geq0$ solves \eqref{eq:riccati_volterra}, then $\widetilde{\psi}=\mathcal{T}\widetilde{\psi}\geq K_f=\psi^{(0)}$, and $\widetilde{\psi}\geq\psi^{(m)}$ implies $\widetilde{\psi}=\mathcal{T}\widetilde{\psi}\geq\mathcal{T}\psi^{(m)}=\psi^{(m+1)}$ and hence $\widetilde{\psi}\geq\psi_f$.

  For uniqueness in the dominated class, let $\widetilde{\psi}$ be as stated. By minimality, $\psi_f\leq\widetilde{\psi}\leq\Theta$, and since $\Theta\in L^1((0,T])$, both solutions are finite outside the Lebesgue-null set $B:=\lbrace r\in(0,T]:\Theta(r)=\infty\rbrace$ (uniformly in $y$). Defining $\delta(r,y):=\widetilde{\psi}(r,y)-\psi_f(r,y)\geq0$ on the complement of $B\times E$ and $\delta:=0$ on $B\times E$, it holds that $\delta\leq\Theta$ pointwise. For $(r,y)$ with $r\notin B$, subtracting the two instances of \eqref{eq:riccati_volterra} and using that the set $B\times E$ is not charged by $ds\otimes m$ together with the Lipschitz bound in \eqref{eq:F_properties}, we obtain
  \[
  \delta(r,y) \leq \lambda_1\overline{z}\int_0^r\int_E K(r-s,y',y)\,\delta(s,y')\,m(dy')\,ds \leq \lambda_1\overline{z}\int_0^r \widetilde{k}(r-s)\,D(s)\,ds,
  \]
  for every measurable majorant $D$ of $\delta$ that is uniform in $y$. For $r\in B$, the same bounds hold trivially since $\delta(r,\cdot)=0$. Iterating $m$ times from $D=\Theta$ yields $\delta(r,y)\leq D_m(r)$ for all $(r,y)$ and $m$, where $D_m:=(\lambda_1\overline{z})^m\,\widetilde{k}^{\star m}\star\Theta$. For $\eta$ with $\lambda_1\overline{z}\widetilde{J}_\eta<1$ we find that
  \[
  \int_0^Te^{-\eta r}D_m(r)\,dr\leq(\lambda_1\overline{z}\widetilde{J}_\eta)^m\int_0^Te^{-\eta r}\Theta(r)\,dr\to0, 
  \]
  and since $\delta\leq\inf_m D_m$, it follows that $\delta(\cdot,y)=0$ Lebesgue-almost everywhere for every $y$.

  The function $\delta$ is measurable, so Tonelli's theorem yields $\delta=0$ for $ds\otimes m(dy)$-almost every $(s,y)$. Therefore the convolution terms on the right-hand sides of the two Riccati-Volterra equations agree for every $(r,y)$ and Lebesgue integration does not charge the product-null set on which the integrands may differ. Substitution into the two equations gives $\widetilde\psi=\psi_f$ everywhere on $(0,T]\times E$. If $\lambda_1=0$, then $\mathcal{T}\psi=K_f$, hence $\psi_f=K_f$.
\end{proof}

\begin{corollary}\label{cor:riccati_pointwise_finite}
In addition to Assumption~\ref{ass:K_transform}, suppose that the dual resolvent $\widetilde\rho$ of \eqref{eq:dual_resolvent} is finite at every $r\in(0,T]$. Then, for every $f\in B_b(E)_+$, the minimal solution $\psi_f$ of \eqref{eq:riccati_volterra} is finite at every $(r,y)\in(0,T]\times E$ and satisfies
\[
0\leq \psi_f(r,y)\leq \lVert f\rVert_\infty\widetilde\rho(r).
\]
It is the unique non-negative Borel solution dominated by a time-only majorant in $L^1((0,T])$.
\end{corollary}
\begin{proof}
The first two claims are immediate from Lemma~\ref{lem:riccati_volterra}. 
\end{proof}

For endpoint evaluations in the stochastic and Lebesgue integrals in what follows, we set
\[
K_f(0,y):=0,\qquad \psi_f(0,y):=0,\qquad y\in E.
\]
The Riccati--Volterra equation remains asserted only for positive time arguments and these endpoint values do not affect any integral, since the Lebesgue measure is atomless, and Proposition~\ref{prop:affine_compensator} gives $\mu(\{t\}\times E\times\Rp)=0$ almost surely for every deterministic $t$. Under the standard pointwise-finite majorants of Corollary~\ref{cor:riccati_pointwise_finite}, $\psi_f$ is finite for every positive lag. The main result of this section uses the pathwise representation of Proposition~\ref{prop:pathwise_identity}.

\begin{remark}\label{rem:time_orientation}
  The variable $r$ in \eqref{eq:riccati_volterra} is the \emph{remaining time} from a jump to the transform horizon. In the transform formula at horizon $t$, the relevant quantity attached to a jump at calendar time $s\leq t$ and location $y$ is $\psi_f(t-s,y)$. This time reversal is the Volterra analogue of the familiar backward Riccati equations of finite-dimensional affine theory.
\end{remark}

\begin{theorem}\label{thm:spatial_transform}
  Let $f\in B_b(E)_+$ and let $\psi_f$ be the minimal solution of \eqref{eq:riccati_volterra} from Lemma~\ref{lem:riccati_volterra}. Then, for every $t\in[0,T]$, it holds that
  \begin{equation}\label{eq:spatial_transform}
  \E\left[\exp\left(-\langle f,V_t\rangle\right)\right] = \exp\left(-\langle f,\phi_t\rangle - \int_0^t\int_E\left(\lambda_0+\lambda_1\phi_s(y)\right)F\bigl(\psi_f(t-s,y)\bigr)\,m(dy)\,ds\right),
  \end{equation}
  where the integral in the exponent is finite.
\end{theorem}
\begin{proof}
  First, for $t=0$, both sides equal $e^{-\langle f,\phi_0\rangle}$. Fix $t\in(0,T]$, set $c_f=\lVert f\rVert_\infty$, and recall that $M=\sup_{(t,x)}\E[V_t(x)]<\infty$.

  We first handle the case where $\lambda_1=0$. In this case, $\psi_f=K_f$ by Lemma~\ref{lem:riccati_volterra}, and the accepted-jump measure $\mu$ is a Poisson random measure with deterministic intensity $\lambda_0\,ds\,m(dy)\,\ell(dz)$. By the pathwise Volterra identity and Tonelli's theorem, it follows that
  \[
  \langle f,V_t\rangle
  =\langle f,\phi_t\rangle
  +\int_{(0,t)\times E\times\Rp}zK_f(t-s,y)\,\mu(ds,dy,dz).
  \]
  The exponential formula for Poisson random measures therefore gives
  \[
  \E[e^{-\langle f,V_t\rangle}]
  =\exp\left(-\langle f,\phi_t\rangle
  -\lambda_0\int_0^t\int_EF(K_f(t-s,y))\,m(dy)\,ds\right),
  \]
  which is \eqref{eq:spatial_transform}. Henceforth assume $\lambda_1>0$.

  \emph{Step 1 (Finiteness and the good event $\Omega_0$).} Since $\int_EK_f(r,w)\,m(dw)\leq c_f\,m(E)\,\widetilde{k}(r)$ and $\psi_f\leq c_f\widetilde{\rho}$ by \eqref{eq:riccati_bound}, the compensator identity from the proof of Proposition~\ref{prop:affine_compensator} applied to deterministic integrands, together with \eqref{eq:V_sup_moment_bound} yields the bounds
  \begin{align}
  \E\left[\int_{(0,t]\times E\times\Rp}z\,K_f(t-u,w)\,\mu(du,dw,dz)\right] &\leq \overline{z}(\lambda_0+\lambda_1M)\,c_f\,m(E)\int_0^t\widetilde{k}(r)\,dr<\infty,\label{eq:fin_a}\\
  \E\left[\int_{(0,t]\times E\times\Rp}z\,\psi_f(t-u,w)\,\mu(du,dw,dz)\right] &\leq \overline{z}(\lambda_0+\lambda_1M)\,c_f\,m(E)\int_0^t\widetilde{\rho}(r)\,dr<\infty,\label{eq:fin_b}
  \end{align}
  and, using $F(\psi_f(t-r,y))\leq \overline{z}c_f\widetilde{\rho}(t-r)$ together with Tonelli's theorem,
  \begin{equation}\label{eq:fin_c}
  \E\left[\int_0^t\int_E V_r(y)\,F\bigl(\psi_f(t-r,y)\bigr)\,m(dy)\,dr\right] \leq M\,\overline{z}\,c_f\,m(E)\int_0^t\widetilde{\rho}(r)\,dr<\infty.
  \end{equation}
  Moreover, let $B:=\lbrace r\in(0,t]:\widetilde{k}(r)=\infty\text{ or }\widetilde{\rho}(r)=\infty\rbrace$, which is a deterministic Lebesgue-null set. Both $\mu(\lbrace t\rbrace\times E\times\Rp)$ and $\mu(\lbrace(u,w,z):t-u\in B\rbrace)$ vanish almost surely because their $\nu$-masses vanish and the compensator identity applies. We then define $\Omega_0$ as the event on which all of the following properties hold:
  \begin{enumerate}
    \item The pathwise Volterra identity \eqref{eq:pathwise_identity} holds simultaneously for every $(r,y)$.
    \item The three pathwise integrals whose expectations appear in \eqref{eq:fin_a}, \eqref{eq:fin_b}, and \eqref{eq:fin_c} are finite.
    \item $\mu(\lbrace t\rbrace\times E\times\Rp)=0$.
    \item $\mu(\lbrace(u,w,z):t-u\in B\rbrace)=0$.
  \end{enumerate}
  The preceding estimates give $\Prob(\Omega_0)=1$.

  \emph{Step 2 (Frozen field and candidate process).} For $s\in[0,t]$ define the frozen forward field
  \begin{equation}\label{eq:frozen_field}
  \xi_s(r,y) := \phi_r(y) + \int_{(0,s]\times E\times\Rp}K(r-u,y,w)\,z\,\mu(du,dw,dz), \qquad (r,y)\in[s,t]\times E,
  \end{equation}
  where the integrand is well defined at $r=u$ by the convention $K(0,\cdot,\cdot)=0$. The field $\xi_s$ is non-negative and, for fixed $s$, jointly measurable in $(\omega,r,y)$. We then set
  \begin{equation}\label{eq:Gamma_def}
  \Gamma_s := \exp\left( -\langle f,\xi_s(t,\cdot)\rangle - \int_s^t\int_E\left(\lambda_0+\lambda_1\xi_s(r,y)\right)F\bigl(\psi_f(t-r,y)\bigr)\,m(dy)\,dr \right) \in[0,1].
  \end{equation}
  On $\Omega_0$ the exponent in \eqref{eq:Gamma_def} is finite for every $s$. Tonelli's theorem and \eqref{eq:fin_a} yields that
  \begin{equation}\label{eq:xi_pairing}
  \langle f,\xi_s(t,\cdot)\rangle = \langle f,\phi_t\rangle + \int_{(0,s]\times E\times\Rp}z\,K_f(t-u,w)\,\mu(du,dw,dz) < \infty.
  \end{equation}
  Additionally, it holds that
  \begin{equation}\label{eq:xi_drift_split}
  \begin{aligned}
  \int_s^t\int_E\xi_s(r,y)F\bigl(\psi_f(t-r,y)\bigr)m(dy)\,dr &= \int_s^t\int_E\phi_r(y)F\bigl(\psi_f(t-r,y)\bigr)m(dy)\,dr \\
  &\quad + \int_{(0,s]\times E\times\Rp}z\,I_s(u,w)\,\mu(du,dw,dz),
  \end{aligned}
  \end{equation}
  where $I_s(u,w):=\int_s^t\int_EK(r-u,y,w)F(\psi_f(t-r,y))\,m(dy)\,dr$. The first term in \eqref{eq:xi_drift_split} is bounded by $\lVert\phi\rVert_\infty\overline{z}c_fm(E)\int_0^t\widetilde\rho(r)\,dr<\infty$, and for the second we note that, by the Riccati--Volterra equation \eqref{eq:riccati_volterra} evaluated at $(t-u,w)$,
  \begin{equation}\label{eq:I_s_bound}
  \lambda_1\,I_s(u,w) \leq \lambda_1\int_u^t\int_EK(r-u,y,w)F(\psi_f(t-r,y))\,m(dy)\,dr = \psi_f(t-u,w)-K_f(t-u,w),
  \end{equation}
  so that $\lambda_1\int z\,I_s\,d\mu\leq\int z\,\psi_f(t-u,w)\,d\mu<\infty$ on $\Omega_0$ by \eqref{eq:fin_b}.

  \emph{Step 3 (Reduction to a stochastic exponential).} We claim that on $\Omega_0$, for all $s\in[0,t]$,
  \begin{equation}\label{eq:Gamma_exponential}
  \Gamma_s = \Gamma_0\,\exp\left( \int_{(0,s]}\chi\,d\mu - \int_{(0,s]}\left(e^{\chi}-1\right)d\nu \right), \qquad \chi(u,w,z):=-z\,\psi_f(t-u,w).
  \end{equation}
  To see this, fix $u\in(0,s]$ with $t-u\notin B$. On $\Omega_0$ this covers $\mu$-almost every atom. The substitution $r=u+v$ combined with the Riccati--Volterra equation \eqref{eq:riccati_volterra} at $(t-u,w)$ give
  \[
  \lambda_1 I_s(u,w) = \bigl[\psi_f(t-u,w)-K_f(t-u,w)\bigr] - \lambda_1\int_0^{s-u}\int_EK(v,y,w)F\bigl(\psi_f\bigl((t-u)-v,y\bigr)\bigr)\,m(dy)\,dv,
  \]
  where the subtraction is justified because all terms are finite for $t-u\notin B$. Integrating against $z\,\mu(du,dw,dz)$ over $(0,s]$, substituting back $r=u+v\in(u,s]$ in the last term, and applying Tonelli's theorem, we obtain
  \[
  \lambda_1\int_{(0,s]}z\,I_s\,d\mu = \int_{(0,s]}z\bigl[\psi_f-K_f\bigr](t-u,w)\,d\mu - \lambda_1\int_0^s\int_E\bigl(V_r(y)-\phi_r(y)\bigr)F\bigl(\psi_f(t-r,y)\bigr)\,m(dy)\,dr,
  \]
  where the last step uses the pathwise identity \eqref{eq:pathwise_identity} (the boundary $\lbrace u=r\rbrace$ contributes a $dr$-null set). Combining this with \eqref{eq:xi_pairing} and \eqref{eq:xi_drift_split} and collecting terms, we find that
  \[
  \begin{aligned}
  \log(\Gamma_s) &= -\langle f,\phi_t\rangle - \int_0^t\int_E(\lambda_0+\lambda_1\phi_r(y))F\bigl(\psi_f(t-r,y)\bigr)m(dy)\,dr \\
  &\qquad - \int_{(0,s]}z\,\psi_f(t-u,w)\,d\mu + \int_0^s\int_E(\lambda_0+\lambda_1V_r(y))F\bigl(\psi_f(t-r,y)\bigr)m(dy)\,dr.
  \end{aligned}
  \]
  Due to the affine form \eqref{eq:affine_compensator} of $\nu$, an application of Tonelli's theorem and the definition \eqref{eq:levy_exponent} of $F$ yields that
  \[
  \int_{(0,s]}\left(1-e^{\chi}\right)d\nu = \int_0^s\int_E(\lambda_0+\lambda_1V_r(y))\,F\bigl(\psi_f(t-r,y)\bigr)\,m(dy)\,dr,
  \]
  which is precisely the claimed identity \eqref{eq:Gamma_exponential}.

  \emph{Step 4 (Martingale argument).} The function $\chi$ is non-positive and Borel measurable, with $\lvert e^{\chi}-1\rvert\leq1\wedge(z\psi_f(t-u,w))$. By the definition of $F$ and the compensator formula, it holds on $\Omega_0$ that
  \[
  \begin{aligned}
  \int_{(0,t]}\lvert e^{\chi}-1\rvert\,d\nu
  &=\int_0^t\int_E(\lambda_0+\lambda_1V_r(y))
   F(\psi_f(t-r,y))\,m(dy)\,dr\\
  &=\lambda_0\int_0^t\int_EF(\psi_f(t-r,y))\,m(dy)\,dr\\
  &\quad+\lambda_1\int_0^t\int_EV_r(y)F(\psi_f(t-r,y))\,m(dy)\,dr<\infty.
  \end{aligned}
  \]
  The first term is finite because $F(q)\leq\overline zq$, \eqref{eq:riccati_bound}, and \eqref{eq:dual_resolvent_L1} apply. The second term is the pathwise integral controlled in \eqref{eq:fin_c}. Likewise, from \eqref{eq:fin_b}, we find that
  \[
  \int_{(0,t]}\lvert e^{\chi}-1\rvert\,d\mu
  \leq\int_{(0,t]}z\psi_f(t-u,w)\,d\mu<\infty.
  \]
  The process $s\mapsto\int_{(0,s]}\lvert e^{\chi}-1\rvert\,d\nu$ is predictable, increasing, and integrable. Its terminal expectation is finite by the deterministic first bound and by the expectation estimate \eqref{eq:fin_c} and thus it is of locally integrable variation. Proposition~II.1.28 of \cite{JacodShiryaev2003} therefore shows that
  \[
  L_s := \int_{(0,s]}\left(e^{\chi}-1\right)d\mu - \int_{(0,s]}\left(e^{\chi}-1\right)d\nu, \qquad s\in[0,t],
  \]
  is a well-defined purely discontinuous local martingale. By Lemma~\ref{lem:distinct_times}, $\mu$ almost surely carries at most one atom at each time. Writing $(u,W_u,Z_u)$ for the atom at a jump time $u$ of $L$, the jumps are $\Delta L_u=e^{\chi(u,W_u,Z_u)}-1\in(-1,0]$, while the $\nu$-part is continuous in time. The absolute convergence proved above justifies the product formula in \cite[Theorem~I.4.61]{JacodShiryaev2003} and gives
  \[
  \mathcal{E}(L)_s = e^{L_s}\prod_{u\leq s}(1+\Delta L_u)e^{-\Delta L_u} = \exp\left( \int_{(0,s]}\chi\,d\mu - \int_{(0,s]}\left(e^{\chi}-1\right)d\nu \right), \qquad s\in[0,t].
  \]
  By Step 3, $\Gamma_s=\Gamma_0\mathcal{E}(L)_s$ for all $s\in[0,t]$ on $\Omega_0$. The constant $\Gamma_0$ is deterministic and strictly positive because its exponent is finite. Since $\Gamma_s\in[0,1]$, we have $\mathcal{E}(L)_s\leq\Gamma_0^{-1}$. For any localizing sequence $(\tau_k)$, dominated convergence therefore yields $\E[\mathcal{E}(L)_t]=\lim_k\E[\mathcal{E}(L)_{t\wedge\tau_k}]=1$. Thus $\mathcal{E}(L)$ is a true martingale and $\E[\Gamma_t]=\Gamma_0$.

  \emph{Step 5 (Identification).} On $\Omega_0$ we have $\mu(\lbrace t\rbrace\times E\times\Rp)=0$, so that, for every $y\in E$,
  \[
  \xi_t(t,y) = \phi_t(y)+\int_{(0,t]}K(t-u,y,w)\,z\,d\mu = \phi_t(y)+\int_{(0,t)}K(t-u,y,w)\,z\,d\mu = V_t(y),
  \]
  by the pathwise identity \eqref{eq:pathwise_identity}. Since the time integral in the exponent of $\Gamma_t$ runs over the empty interval $(t,t]$, we obtain $\Gamma_t=\exp(-\langle f,V_t\rangle)$ almost surely. On the other hand, $\xi_0=\phi$, so $\Gamma_0$ equals the right-hand side of \eqref{eq:spatial_transform}, and $\E[\Gamma_t]=\Gamma_0$ which completes the proof.
\end{proof}

\begin{remark}\label{rem:affine_structure}
The log-Laplace functional in \eqref{eq:spatial_transform} is affine in $\phi$, while $\psi_f$ depends only on $(K,\ell,\lambda_1,f)$. Moreover, $0\leq\Gamma\leq1$ in the proof, so the stochastic exponential is a true martingale without any exponential-moment assumption on $\ell$.
\end{remark}

\begin{example}[L\'evy--Khintchine formula for ambit fields]\label{ex:levy_khintchine}
  If $\lambda_1=0$, then $\psi_f=K_f$ by Lemma~\ref{lem:riccati_volterra}, and \eqref{eq:spatial_transform} becomes
  \[
  \E\left[e^{-\langle f,V_t\rangle}\right] = \exp\left(-\langle f,\phi_t\rangle-\lambda_0\int_0^t\int_EF\bigl(K_f(s,y)\bigr)\,m(dy)\,ds\right),
  \]
  after the substitution $s\mapsto t-s$. This is the classical L\'evy--Khintchine identity (see \cite{Rosinski1989}) for the L\'evy-driven convolution of Example~\ref{ex:ambit_field}, as can be verified directly from the exponential formula for the Poisson random measure $\mu$ with deterministic intensity $\lambda_0\,ds\,m(dy)\,\ell(dz)$.
\end{example}

\begin{example}[Affine Volterra jump processes and CBI-type processes]\label{ex:transform_singleton}
  In the setting of Example~\ref{ex:affine_pure_jump} ($E=\lbrace e\rbrace$, $m=\delta_e$, scalar kernel $K\in L^1((0,T])$), Assumption~\ref{ass:K_transform} holds automatically with $\widetilde{k}=k=K$, and for $f\geq0$ the Riccati--Volterra equation reads
  \[
  \psi_f(r) = f\,K(r) + \lambda_1\int_0^rK(r-s)\,F\bigl(\psi_f(s)\bigr)\,ds,
  \]
  with transform $\E[e^{-fV_t}] = \exp(-f\phi_t-\int_0^t(\lambda_0+\lambda_1\phi_s)F(\psi_f(t-s))\,ds)$, in line with the affine Volterra jump processes of \cite{BondiLivieriPulido2024}. For $K\equiv1$ and $\phi\equiv v_0\geq0$, the positive-time solution is the restriction of the continuous function $\widehat\psi_f$ satisfying
  \[
  \widehat\psi_f'(r)=\lambda_1F(\widehat\psi_f(r)),
  \qquad \widehat\psi_f(0)=f.
  \]
  The auxiliary convention $\psi_f(0)=0$ used in stochastic integrals does not change this right limit and the transform becomes
  \[
  \E\left[e^{-fV_t}\right] = \exp\left(-v_0\,\widehat\psi_f(t)-\lambda_0\int_0^tF\bigl(\widehat\psi_f(s)\bigr)\,ds\right),
  \]
  where we used $f+\lambda_1\int_0^tF(\widehat\psi_f(s))\,ds=\widehat\psi_f(t)$. In the standard CBI sign convention, the pure-jump branching and immigration mechanisms are $\Psi(q)=-\lambda_1F(q)$ and $\Phi(q)=\lambda_0F(q)$, respectively. The forward equation above is the corresponding backward Laplace equation and this recovers the finite-variation CBI construction (compare with \cite{DawsonLi2012}).
\end{example}

\begin{example}[Finite-rank exponential kernels]\label{ex:finite_rank}
  Let
  \[
  K(r,x,y)=e^{-\varrho r}\sum_{i,j=1}^db_i(x)\,\kappa_{ij}\,a_j(y),
  \]
  where $\varrho>0$, the functions $a_j,b_i:E\to\Rp$ are bounded and Borel, and the coefficients satisfy $\kappa_{ij}\geq0$. Both the admissibility criteria and Assumption~\ref{ass:K_transform} hold, and the Riccati--Volterra equation \eqref{eq:riccati_volterra} reduces to a $d$-dimensional system. Consider
  \[
  q_j(r) = e^{-\varrho r}c_j + \lambda_1\int_{0}^{r}e^{-\varrho(r-s)}\sum_{i=1}^{d}\kappa_{ij}\int_{E}b_i(y)F\Bigl(\sum_{k=1}^dq_k(s)a_k(y)\Bigr)\,m(dy)ds, \quad c_j=\sum_{i=1}^d\kappa_{ij}\langle f,b_i\rangle.
  \]
  Since $F$ is Lipschitz and the $a_k$, $b_i$ are bounded, this integral system has a unique bounded componentwise non-negative solution $q=(q_1,\ldots,q_d)$ on $[0,T]$, which is continuous by Picard iteration. The monotone iterates of Lemma~\ref{lem:riccati_volterra} preserve the finite-rank form with componentwise non-decreasing coefficient iterates that, by induction, are dominated by $q$, so their limit is a bounded solution of the same system and coincides with $q$. Hence $\psi_f(r,y)=\sum_{j=1}^dq_j(r)\,a_j(y)$ for $r\in(0,T]$. The initial values $q_j(0)=c_j$ describe the right limit of this positive-time solution at $r=0$, while the endpoint convention $\psi_f(0,\cdot)=0$ used in stochastic integrals is unaffected, exactly as in Example~\ref{ex:transform_singleton}. Multiplication by $e^{\varrho r}$ followed by differentiation shows that the integral system is equivalent to the generalized Riccati system
  \[
  q_j'(r) = -\varrho\,q_j(r) + \lambda_1\sum_{i=1}^d\kappa_{ij}\int_Eb_i(y)\,F\Bigl(\sum_{k=1}^dq_k(r)a_k(y)\Bigr)\,m(dy), \qquad q_j(0)=c_j,
  \]
  which connects the transform formula \eqref{eq:spatial_transform} to the exponential-affine transforms of classical finite-dimensional affine jump processes \cite{DuffiePanSingleton2000,DuffieFilipovicScachermayer2003,ErraisGieseckeGoldberg2010}.
\end{example}

\section{Space-time and conditional affine transforms}\label{sec:spacetime_transform}
The purpose of this section is to derive the Laplace transform of 
\[
\int_0^T\langle g_s,V_s\rangle\,ds,\quad g\in B_b([0,T]\times E)_+.
\]
Throughout the section, Assumption~\ref{ass:K_transform} remains in force and we note that the Laplace transform depends only on the $\mathcal{H}_\eta$-equivalence class of $V$. Define the future dual source $G_g$ as
\begin{equation}\label{eq:future_source}
G_g(s,y) := \int_s^T\int_E g_r(x)\,K(r-s,x,y)\,m(dx)\,dr, \qquad (s,y)\in[0,T]\times E,
\end{equation}
which is Borel measurable and satisfies
\begin{equation}\label{eq:C_g_bound}
0\leq G_g(s,y) \leq \lVert g\rVert_\infty\int_0^{T}\widetilde{k}(v)\,dv =: C_g < \infty,
\end{equation}
since $\int_Eg_r(x)K(r-s,x,y)\,m(dx)\leq\lVert g\rVert_\infty\,\widetilde{k}(r-s)$. The associated Riccati--Volterra equation now runs backward from the terminal time $T$.

\begin{lemma}\label{lem:backward_riccati}
  Let $g\in B_b([0,T]\times E)_+$. Then there exists a bounded Borel function $\psi_g:[0,T]\times E\to\Rp$ satisfying
  \begin{equation}\label{eq:backward_riccati}
  \psi_g(s,y) = G_g(s,y) + \lambda_1\int_s^T\int_EK(r-s,y',y)\,F\bigl(\psi_g(r,y')\bigr)\,m(dy')\,dr, \qquad (s,y)\in[0,T]\times E,
  \end{equation}
  and $\psi_g$ is the unique bounded non-negative Borel solution of \eqref{eq:backward_riccati}. For any $\eta\geq0$ with $\lambda_1\overline{z}\widetilde{J}_\eta<1$ it satisfies $0\leq\psi_g\leq e^{\eta T}C_g/(1-\lambda_1\overline{z}\widetilde{J}_\eta)$. Moreover, for every $t\in[0,T]$ the restriction of $\psi_g$ to $[t,T]\times E$ depends on $g$ only through its restriction to $[t,T]\times E$, and solves \eqref{eq:backward_riccati} with $[t,T]$ in place of $[0,T]$.
\end{lemma}
\begin{proof}
  We repeat the monotone iteration of Lemma~\ref{lem:riccati_volterra} in reversed time. Denote by $\mathcal{T}_g$ the operator mapping a Borel function $\psi:[0,T]\times E\to[0,\infty]$ to the right-hand side of \eqref{eq:backward_riccati}. As previously, $\mathcal{T}_g$ preserves Borel measurability and is monotone, and the iterates 
  \[
  \psi^{(0)}:=G_g, \quad \psi^{(m+1)}:=\mathcal{T}_g\psi^{(m)},
  \]
  increase pointwise to a solution $\psi_g$ of \eqref{eq:backward_riccati} with values in $[0,\infty]$ by monotone convergence.

  For the bound, fix $\eta\geq0$ with $\lambda_1\overline{z}\widetilde{J}_\eta<1$ and define the deterministic majorants $b_0:\equiv C_g$ and $b_{m+1}(s):=C_g+\lambda_1\overline{z}\int_s^T\widetilde{k}(r-s)\,b_m(r)\,dr$. Using $F(q)\leq\overline{z}q$ and $\int_EK(r-s,y',y)\,m(dy')\leq\widetilde{k}(r-s)$, an induction gives $\psi^{(m)}\leq b_m$ pointwise for every $m$. Writing $B_m:=\sup_{s\in[0,T]}e^{-\eta(T-s)}b_m(s)$ and inserting $e^{\eta(T-r)}=e^{\eta(T-s)}e^{-\eta(r-s)}$, we obtain
  \[
  e^{-\eta(T-s)}b_{m+1}(s) \leq C_g + \lambda_1\overline{z}\,B_m\int_s^Te^{-\eta(r-s)}\widetilde{k}(r-s)\,dr \leq C_g+\lambda_1\overline{z}\widetilde{J}_\eta B_m,
  \]
  so that $B_{m+1}\leq C_g+\lambda_1\overline{z}\widetilde{J}_\eta B_m$ and hence $\sup_mB_m\leq C_g/(1-\lambda_1\overline{z}\widetilde{J}_\eta)$. Letting $m\to\infty$ yields the stated bound for $\psi_g$ and in particular $\psi_g$ is bounded and finite everywhere.

  For uniqueness, let $\widetilde{\psi}$ be another bounded non-negative Borel solution and set $\delta:=\lvert\widetilde{\psi}-\psi_g\rvert\leq\Theta$ for a constant $\Theta<\infty$. Subtracting the two equations (all terms are finite) and using the Lipschitz property of $F$ from \eqref{eq:F_properties} gives $\delta(s,y)\leq\lambda_1\overline{z}\int_s^T\widetilde{k}(r-s)\,D(r)\,dr$ for any deterministic majorant $D$ of $\delta$ that is uniform in $y$. Defining $D_0:\equiv\Theta$ and $D_{m+1}(s):=\lambda_1\overline{z}\int_s^T\widetilde{k}(r-s)D_m(r)\,dr$, induction shows both $\delta\leq D_m$ and $D_m(s)\leq\Theta\,(\lambda_1\overline{z}\widetilde{J}_\eta)^m\,e^{\eta(T-s)}$ for all $m$, so $\delta\equiv 0$.

  Finally, equation \eqref{eq:backward_riccati} on $[t,T]\times E$ involves $\psi_g$ and $g$ only through their values on $[t,T]\times E$, so the restriction statement follows from applying the existence and uniqueness assertions on the interval $[t,T]$ and noting that the restriction of $\psi_g$ is a bounded solution there.
\end{proof}

\begin{theorem}\label{thm:spacetime_transform}
  Let $g\in B_b([0,T]\times E)_+$ and let $\psi_g$ be the unique bounded solution of \eqref{eq:backward_riccati}. Then
  \begin{equation}\label{eq:spacetime_transform}
  \begin{aligned}
  &\E\left[\exp\left(-\int_0^T\langle g_s,V_s\rangle\,ds\right)\right]\\
  &\quad=\exp\left(
  -\int_0^T\langle g_s,\phi_s\rangle\,ds
  -\int_0^T\int_E(\lambda_0+\lambda_1\phi_s(y))
  F(\psi_g(s,y))\,m(dy)\,ds\right),
  \end{aligned}
  \end{equation}
  where the integrals in the exponent are finite.
\end{theorem}
\begin{proof}
  The proof follows Theorem~\ref{thm:spatial_transform}. Note that boundedness of $\psi_g$ simplifies the integrability step.

  Since $\E[\int_{(0,T]\times E\times\Rp}z\,d\mu]=\overline{z}\int_0^T\int_E(\lambda_0+\lambda_1\E[V_u(w)])\,m(dw)\,du\leq\overline{z}(\lambda_0+\lambda_1M)Tm(E)<\infty$ by the compensator identity, the integrals
  \[
  \int_{(0,T]\times E\times\Rp} z\,G_g(u,w)\,d\mu \leq C_g\int z\,d\mu, \qquad \int_{(0,T]\times E\times\Rp} z\,\psi_g(u,w)\,d\mu \leq C_\psi\int z\,d\mu
  \]
  are almost surely finite, where $C_\psi:=\sup\psi_g<\infty$. Consequently, the same holds for 
  \[
  \int_0^T\int_EV_r(y)F(\psi_g(r,y))\,m(dy)\,dr\leq \overline{z}C_\psi\int_0^T\int_EV_r(y)\,m(dy)\,dr,
  \]
  whose expectation is bounded by $\overline{z}C_\psi MTm(E)$. Let $\Omega_0$ be the almost-sure event on which these three integrals are finite and the pathwise identity \eqref{eq:pathwise_identity} holds for all $(r,y)$.

  For $s\in[0,T]$ define $\xi_s(r,y)$ by \eqref{eq:frozen_field} for $(r,y)\in[s,T]\times E$, and set
  \[
  \begin{aligned}
  \Gamma_s := \exp\biggl(-\int_0^s\langle g_r,V_r\rangle\,dr &- \int_s^T\langle g_r,\xi_s(r,\cdot)\rangle\,dr \\
  &- \int_s^T\int_E\left(\lambda_0+\lambda_1\xi_s(r,y)\right)F\bigl(\psi_g(r,y)\bigr)\,m(dy)\,dr\biggr)\in[0,1].
  \end{aligned}
  \]
  On $\Omega_0$, all three exponent terms are finite. To see this, apply Tonelli's theorem and \eqref{eq:pathwise_identity}, to obtain
  \begin{equation}\label{eq:st_first_terms}
  \begin{aligned}
  \int_0^s\langle g_r,V_r\rangle\,dr + \int_s^T\langle g_r,\xi_s(r,\cdot)\rangle\,dr &= \int_0^T\langle g_r,\phi_r\rangle\,dr \\
  &\quad + \int_{(0,s]\times E\times\Rp}z\,G_g(u,w)\,\mu(du,dw,dz).
  \end{aligned}
  \end{equation}
  Splitting $\xi_s=\phi+\int K\,z\,d\mu$ and using the backward Riccati--Volterra equation \eqref{eq:backward_riccati} at $(u,w)$, we find that
  \[
  \begin{aligned}
  \lambda_1\int_s^T\int_EK(r-u,y,w)F\bigl(\psi_g(r,y)\bigr)m(dy)\,dr &= \bigl[\psi_g(u,w)-G_g(u,w)\bigr] \\
  &\quad - \lambda_1\int_u^s\int_EK(r-u,y,w)F\bigl(\psi_g(r,y)\bigr)m(dy)\,dr,
  \end{aligned}
  \]
  where all terms are finite because $\psi_g$ is bounded, and re-applying Tonelli's theorem together with \eqref{eq:pathwise_identity} exactly as in Step~3 of the proof of Theorem~\ref{thm:spatial_transform}, we obtain
  \[
  \begin{aligned}
  \int_s^T\int_E\left(\lambda_0+\lambda_1\xi_s(r,y)\right)F\bigl(\psi_g(r,y)\bigr)m(dy)\,dr &= \int_s^T\int_E(\lambda_0+\lambda_1\phi_r(y))F\bigl(\psi_g(r,y)\bigr)m(dy)\,dr \\
  &\quad + \int_{(0,s]}z\bigl[\psi_g-G_g\bigr](u,w)\,d\mu \\
  &\quad - \lambda_1\int_0^s\int_E\bigl(V_r(y)-\phi_r(y)\bigr)F\bigl(\psi_g(r,y)\bigr)m(dy)\,dr.
  \end{aligned}
  \]
  Adding this to \eqref{eq:st_first_terms} and collecting terms yields, on $\Omega_0$ and for all $s\in[0,T]$, the representation
  \begin{equation}\label{eq:st_Gamma_exponential}
  \Gamma_s = \Gamma_0\exp\left(\int_{(0,s]}\chi_g\,d\mu - \int_{(0,s]}\left(e^{\chi_g}-1\right)d\nu\right), \qquad \chi_g(u,w,z):=-z\,\psi_g(u,w),
  \end{equation}
  with $\Gamma_0 = \exp(-\int_0^T\langle g_r,\phi_r\rangle\,dr-\int_0^T\int_E(\lambda_0+\lambda_1\phi_r(y))F(\psi_g(r,y))\,m(dy)\,dr)$, where we again used $\int_{(0,s]}(1-e^{\chi_g})\,d\nu = \int_0^s\int_E(\lambda_0+\lambda_1V_r(y))F(\psi_g(r,y))\,m(dy)\,dr$. The integrability inputs for the stochastic-exponential argument are
  \[
  \int|e^{\chi_g}-1|\,d\mu\leq C_\psi\int z\,d\mu<\infty,
  \]
  and
  \[
  \int|e^{\chi_g}-1|\,d\nu
  =\int_0^T\int_E(\lambda_0+\lambda_1V_s(y))F(\psi_g(s,y))\,m(dy)\,ds<\infty.
  \]
  Repeating the stochastic-exponential argument from the proof of Theorem~\ref{thm:spatial_transform} shows that the right-hand side of \eqref{eq:st_Gamma_exponential} is $\Gamma_0\mathcal{E}(L)_s$ for $L:=(e^{\chi_g}-1)\star(\mu-\nu)$, and $\mathcal E(L)_s\leq\Gamma_0^{-1}$. Thus $\Gamma$ is a bounded $\Fcal_s^N$-martingale and $\E[\Gamma_T]=\Gamma_0$. Since $\Gamma_T=\exp(-\int_0^T\langle g_r,V_r\rangle\,dr)$, this is \eqref{eq:spacetime_transform}.
\end{proof}

\begin{remark}\label{rem:Gamma_martingale}
  The proof of Theorem~\ref{thm:spacetime_transform} shows more than the identity \eqref{eq:spacetime_transform}. Indeed, the process
  \[
  \Gamma_s = \exp\left(-\int_0^s\langle g_r,V_r\rangle\,dr - \int_s^T\langle g_r,\xi_s(r,\cdot)\rangle\,dr - \int_s^T\int_E\left(\lambda_0+\lambda_1\xi_s(r,y)\right)F\bigl(\psi_g(r,y)\bigr)\,m(dy)\,dr\right),
  \]
  is a bounded $\Fcal_s^N$-martingale on $[0,T]$, which is the key to the conditional transform below.
\end{remark}

\subsection{The conditional transform from the frozen field}\label{subsec:conditional_transform}

For $t\in[0,T]$, define the \emph{frozen forward field}
\begin{equation}\label{eq:forward_curve_lift}
X_t(r,x) := \phi_r(x) + \int_{(0,t]\times E\times\Rp}K(r-u,x,w)\,z\,\mu(du,dw,dz), \qquad (r,x)\in(t,T]\times E.
\end{equation}
The field $X_t$ records the driver and jumps observed by time $t$. It is the input to the future field. For fixed $t$, $X_t=\xi_t$ in \eqref{eq:frozen_field} and is $\Fcal_t^N\otimes\Bcal((t,T])\otimes\Ecal$-measurable by the localization argument of Lemma~\ref{lem:predictability}.

\begin{lemma}\label{lem:conditional_mean}
  Fix $t\in[0,T)$ and define for non-negative $\Fcal_t^N\otimes\Bcal((t,T])\otimes\Ecal$-measurable fields $H$ on $(t,T]\times E$ with values in $[0,\infty]$ the following quantities
  \begin{equation}\label{eq:K_t_operator}
  \begin{aligned}
  b_t(r,x) &:= \overline z\lambda_0\int_t^r\int_EK(r-s,x,y)\,m(dy)\,ds,\\
  (\mathcal{K}_tH)(r,x) &:= \overline z\lambda_1\int_t^r\int_EK(r-s,x,y)\,H(s,y)\,m(dy)\,ds,
  \end{aligned}
  \end{equation}
  and set
  \begin{equation}\label{eq:h_t_neumann}
  h_t := \sum_{n=0}^{\infty}\mathcal{K}_t^n\left(X_t+b_t\right).
  \end{equation}
  Then the following hold.
  \begin{enumerate}
    \item[(i)] The field $h_t$ is $\Fcal_t^N\otimes\Bcal((t,T])\otimes\Ecal$-measurable with values in $[0,\infty]$ and satisfies, for every $(\omega,r,x)\in\Omega\times(t,T]\times E$, the relation
    \begin{equation}\label{eq:conditional_mean_volterra}
    h_t(r,x)=X_t(r,x)+\overline z\int_t^r\int_EK(r-s,x,y)
    \bigl(\lambda_0+\lambda_1h_t(s,y)\bigr)m(dy)\,ds.
    \end{equation}
    \item[(ii)] For every point $(r,x)\in(t,T]\times E$, almost surely $h_t(r,x)=\E[V_r(x)\mid\Fcal_t^N]$. $\E[h_t(r,x)]\leq M$ for every $(r,x)\in(t,T]\times E$, and $h_t(r,x)<\infty$ almost surely for each fixed $(r,x)$.
    \item[(iii)] The field $h_t$ is the unique solution of \eqref{eq:conditional_mean_volterra} in the following class: if $\widetilde{h}$ is non-negative, $\Fcal_t^N\otimes\Bcal((t,T])\otimes\Ecal$-measurable, has uniformly bounded first moments $\sup_{(r,x)}\E[\widetilde{h}(r,x)]<\infty$, and solves \eqref{eq:conditional_mean_volterra} almost surely for every $(r,x)\in(t,T]\times E$, then $\widetilde{h}(r,x)=h_t(r,x)$ almost surely for every $(r,x)\in(t,T]\times E$.
  \end{enumerate}
\end{lemma}
\begin{proof}
  \emph{(i).} The function $b_t$ is Borel measurable by Tonelli's theorem, $X_t$ is measurable as shown above, and $\mathcal{K}_t$ maps non-negative $\Fcal_t^N\otimes\Bcal((t,T])\otimes\Ecal$-measurable fields to fields of the same type, again by Tonelli's theorem. Hence every partial sum of \eqref{eq:h_t_neumann}, and therefore $h_t$, is measurable. Applying $\mathcal{K}_t$ to \eqref{eq:h_t_neumann} and interchanging $\mathcal{K}_t$ with the increasing limit of the partial sums by monotone convergence yields $\mathcal{K}_th_t=\sum_{n\geq1}\mathcal{K}_t^n(X_t+b_t)$, whence $X_t+b_t+\mathcal{K}_th_t=h_t$ pointwise in $[0,\infty]$, which is \eqref{eq:conditional_mean_volterra}.

  \emph{(ii).} Fix $(r,x)\in(t,T]\times E$. On the almost-sure event of Proposition~\ref{prop:pathwise_identity}, splitting \eqref{eq:pathwise_identity} at $t$ gives, simultaneously for all $(s,y)\in(t,T]\times E$,
  \begin{equation}\label{eq:split_identity}
  V_s(y) = X_t(s,y) + \int_{(t,s)\times E\times\Rp}K(s-u,y,w)\,z\,\mu(du,dw,dz).
  \end{equation}
  We use the following conditional form of the compensator identity: for every non-negative Borel function $H$ on $(t,T]\times E\times\Rp$ and every bounded non-negative $\Fcal_t^N$-measurable $G$, the integrand $G(\omega)\,\mathbf{1}_{(t,r)}(u)\,H(u,w,z)$ is $\Pcal\otimes\Bcal(\Rp)$-measurable, so the compensator identity from the proof of Proposition~\ref{prop:affine_compensator} gives $\E[G\int_{(t,r)}H\,d\mu]=\E[G\int_{(t,r)}H\,d\nu]$ and hence
  \begin{equation}\label{eq:conditional_compensator}
  \E\left[\int_{(t,r)\times E\times\Rp}H\,d\mu\Bigm\vert\Fcal_t^N\right] = \E\left[\int_{(t,r)\times E\times\Rp}H\,d\nu\Bigm\vert\Fcal_t^N\right] \qquad\text{almost surely, in }[0,\infty].
  \end{equation}
  Define the restricted kernel iterates $K_t^{[1]}(r,s,x,y):=K(r-s,x,y)$ and
  \[
  K_t^{[n+1]}(r,s,x,y):=\int_s^r\int_EK(r-v,x,w)\,K_t^{[n]}(v,s,w,y)\,m(dw)\,dv, \qquad t<s<r\leq T,
  \]
  so that $(\mathcal{K}_t^nH)(r,x)=(\overline z\lambda_1)^n\int_t^r\int_EK_t^{[n]}(r,s,x,y)\,H(s,y)\,m(dy)\,ds$. By Tonelli's theorem, the same iterates satisfy the mirrored recursion
  \[
  K_t^{[n+1]}(r,u,x,w)=\int_u^r\int_EK_t^{[n]}(r,s,x,y)\,K(s-u,y,w)\,m(dy)\,ds,
  \]
  and by induction $\int_EK_t^{[n]}(r,s,x,y)\,m(dy)\leq k^{\star n}(r-s)$, cf.~\eqref{eq:convolution_estimates}. We claim that for every $N\in\N_0$, almost surely,
  \begin{equation}\label{eq:conditional_iteration}
  \E\left[V_r(x)\mid\Fcal_t^N\right] = \sum_{n=0}^{N}\mathcal{K}_t^n(X_t+b_t)(r,x) + R_N,
  \end{equation}
  where $R_N := (\overline z\lambda_1)^{N+1}\,\E[\int_t^r\int_EK_t^{[N+1]}(r,s,x,y)\,V_s(y)\,m(dy)\,ds\mid\Fcal_t^N]$. For $N=0$, take conditional expectations in \eqref{eq:split_identity} at $(s,y)=(r,x)$: the summand $X_t(r,x)$ is $\Fcal_t^N$-measurable, and \eqref{eq:conditional_compensator} with $H(u,w,z)=K(r-u,x,w)\,z$ turns the stochastic integral into $\E[\overline z\int_t^r\int_EK(r-s,x,y)(\lambda_0+\lambda_1V_s(y))\,m(dy)\,ds\mid\Fcal_t^N] = b_t(r,x)+R_0$. For the induction step, insert \eqref{eq:split_identity} into $R_N$ and reorder the resulting double integral by Tonelli's theorem, pathwise: the $X_t$-part contributes the $\Fcal_t^N$-measurable term $\mathcal{K}_t^{N+1}X_t(r,x)$, while the $\mu$-part becomes the integral of the deterministic function $H(u,w,z)=(\overline z\lambda_1)^{N+1}K_t^{[N+2]}(r,u,x,w)\,z$ against $\mu$ over $(t,r)\times E\times\Rp$, by the mirrored recursion. Applying \eqref{eq:conditional_compensator} to it, the $\lambda_0$-part of the compensator contributes $\mathcal{K}_t^{N+1}b_t(r,x)$, again by the mirrored recursion, and the $\lambda_1$-part is $R_{N+1}$.

  Fix $\eta\geq0$ with $\theta:=\overline z\lambda_1\int_0^Te^{-\eta v}k(v)\,dv<1$. By the tower property, Tonelli's theorem, \eqref{eq:V_sup_moment_bound} and \eqref{eq:convolution_estimates},
  \[
  \E[R_N]\leq M\,(\overline z\lambda_1)^{N+1}\int_0^Tk^{\star(N+1)}(v)\,dv\leq M\,e^{\eta T}\theta^{N+1}\xrightarrow[N\to\infty]{}0.
  \]
  Fixing versions of the countably many conditional expectations in \eqref{eq:conditional_iteration} on a common almost-sure event, the partial sums increase to $h_t(r,x)$, so the remainders $R_N$ decrease to a non-negative limit whose expectation vanishes. It follows that $R_N\downarrow0$ almost surely and $\E[V_r(x)\mid\Fcal_t^N]=h_t(r,x)$ almost surely. The moment bound follows from $\E[h_t(r,x)]=\E[V_r(x)]\leq M$, and it implies $h_t(r,x)<\infty$ almost surely for each fixed $(r,x)$.

  \emph{(iii).} Let $\widetilde{h}$ be as stated and set $\delta(r,x):=\E[\lvert\widetilde{h}(r,x)-h_t(r,x)\rvert]$, which is measurable by Tonelli's theorem and bounded by $\sup_{(r,x)}\E[\widetilde{h}(r,x)]+M$. For fixed $(r,x)$, the integrals on the right-hand side of \eqref{eq:conditional_mean_volterra} for $\widetilde{h}$ and $h_t$ are almost surely finite because their expectations are finite, so subtracting the two equations and taking expectations yields
  \[
  \delta(r,x)\leq\overline z\lambda_1\int_t^r\int_EK(r-s,x,y)\,\delta(s,y)\,m(dy)\,ds, \qquad (r,x)\in(t,T]\times E.
  \]
  With $\lVert\delta\rVert_{\eta,\infty}:=\sup_{(r,x)\in(t,T]\times E}e^{-\eta r}\delta(r,x)<\infty$, this gives $\lVert\delta\rVert_{\eta,\infty}\leq\overline z\lambda_1J_\eta(T)\,\lVert\delta\rVert_{\eta,\infty}$ with $J_\eta(T)$ as in \eqref{eq:J_def}, and choosing $\eta$ with $\overline z\lambda_1J_\eta(T)<1$ forces $\delta\equiv0$.
\end{proof}

Thus the conditional mean is a measurable functional of $X_t$. The conditional Laplace transform has the same forward-state dependence.

\begin{theorem}\label{thm:conditional_transform}
  Let $g\in B_b([0,T]\times E)_+$, let $\psi_g$ be the unique bounded solution of \eqref{eq:backward_riccati}, and let $t\in[0,T]$. Then, almost surely,
  \begin{equation}\label{eq:conditional_transform}
  \begin{aligned}
  \E\left[\exp\left(-\int_t^T\langle g_s,V_s\rangle\,ds\right)\Bigm\vert\Fcal_t^N\right] = \exp\biggl(-\int_t^T\langle g_r,X_t(r,\cdot)\rangle\,dr \qquad& \\
  - \int_t^T\int_E\left(\lambda_0+\lambda_1X_t(r,y)\right)F\bigl(\psi_g(r,y)\bigr)\,m(dy)\,dr\biggr)&.
  \end{aligned}
  \end{equation}
  By Lemma~\ref{lem:backward_riccati}, the right-hand side depends on $g$ only through its restriction to $[t,T]\times E$.
\end{theorem}
\begin{proof}
  Let $\Gamma$ be the bounded $\Fcal_s^N$-martingale of Remark~\ref{rem:Gamma_martingale}, so that $\E[\Gamma_T\,\vert\,\Fcal_t^N]=\Gamma_t$ almost surely. The random variable $A:=\int_0^t\langle g_r,V_r\rangle\,dr$ is $\Fcal_t^N$-measurable: the restriction of the predictable field $V$ to $\Omega\times[0,t]\times E$ is $\Fcal_t^N\otimes\Bcal([0,t])\otimes\Ecal$-measurable, and joint measurability plus Tonelli's theorem apply; moreover $A<\infty$ almost surely. Writing $\Gamma_T=e^{-A}\exp(-\int_t^T\langle g_r,V_r\rangle\,dr)$ and pulling out the bounded $\Fcal_t^N$-measurable factor $e^{-A}$,
  \[
  \begin{aligned}
  e^{-A}\,\E\left[\exp\left(-\int_t^T\langle g_r,V_r\rangle\,dr\right)\Bigm\vert\Fcal_t^N\right] &= \Gamma_t \\
  &= e^{-A}\exp\biggl(-\int_t^T\langle g_r,X_t(r,\cdot)\rangle\,dr \\
  &\qquad\qquad -\int_t^T\int_E(\lambda_0+\lambda_1X_t(r,y))F\bigl(\psi_g(r,y)\bigr)m(dy)\,dr\biggr),
  \end{aligned}
  \]
  using $\xi_t=X_t$. Multiplying by $e^{A}\in(0,\infty)$ yields \eqref{eq:conditional_transform}.
\end{proof}

\begin{remark}\label{rem:random_phi_connection}
Conditionally on $\Fcal_t^N$, the post-$t$ equation has random driver $X_t$. Thus \eqref{eq:conditional_transform} is the conditional counterpart of \eqref{eq:spacetime_transform}.
\end{remark}

\begin{corollary}\label{cor:exponential_conditional}
  Let $K(r,x,y)=e^{-\varrho r}\kappa(x,y)$ with $\varrho\geq0$ and a bounded Borel function $\kappa:E\times E\to\Rp$, fix $t\in[0,T)$, and assume the driver has the exponential flow property $\phi_r(x)=e^{-\varrho(r-t)}\phi_t(x)$ for all $(r,x)\in[t,T]\times E$. Define $(\kappa^*p)(y):=\int_E\kappa(x,y)\,p(x)\,m(dx)$ and let $p:[t,T]\times E\to\Rp$ be given by
  \begin{equation}\label{eq:p_mild}
  p_s(x) := \int_s^Te^{-\varrho(r-s)}\Bigl(g_r(x)+\lambda_1F\bigl(\psi_g(r,x)\bigr)\Bigr)\,dr, \qquad (s,x)\in[t,T]\times E.
  \end{equation}
  Then $p$ is the unique bounded non-negative Borel solution of the mild backward Riccati equation
  \begin{equation}\label{eq:p_riccati}
  p_s(x) = \int_s^Te^{-\varrho(r-s)}\Bigl(g_r(x)+\lambda_1F\bigl((\kappa^*p_r)(x)\bigr)\Bigr)\,dr,
  \end{equation}
  it satisfies $\psi_g(s,y)=(\kappa^*p_s)(y)$ on $[t,T]\times E$, and almost surely
  \begin{equation}\label{eq:exponential_conditional_transform}
  \begin{aligned}
  \E\left[\exp\left(-\int_t^T\langle g_s,V_s\rangle\,ds\right)\Bigm\vert\Fcal_t^N\right]
  &=\exp\bigl(-A_t-\langle p_t,V_t\rangle\bigr),\\
  A_t&:=\lambda_0\int_t^T\int_E
  F\bigl((\kappa^*p_s)(y)\bigr)\,m(dy)\,ds.
  \end{aligned}
  \end{equation}
\end{corollary}
\begin{proof}
  Inserting $K(r-s,x,y)=e^{-\varrho(r-s)}\kappa(x,y)$ into \eqref{eq:future_source} and \eqref{eq:backward_riccati} and interchanging integrals by Tonelli's theorem shows that, on $[t,T]\times E$,
  \[
  \psi_g(s,y) = \int_E\kappa(x,y)\left[\int_s^Te^{-\varrho(r-s)}\Bigl(g_r(x)+\lambda_1F\bigl(\psi_g(r,x)\bigr)\Bigr)dr\right]m(dx) = (\kappa^*p_s)(y)
  \]
  with $p$ as in \eqref{eq:p_mild}. In particular, $p$ is bounded, non-negative, Borel, and solves \eqref{eq:p_riccati}. Conversely, if $\widetilde{p}$ is any bounded non-negative Borel solution of \eqref{eq:p_riccati}, then $\widetilde{\psi}(s,y):=(\kappa^*\widetilde{p}_s)(y)$ is a bounded non-negative solution of \eqref{eq:backward_riccati} on $[t,T]\times E$ (with $g$ restricted accordingly), so $\widetilde{\psi}=\psi_g$ there by Lemma~\ref{lem:backward_riccati}, and then \eqref{eq:p_riccati} returns $\widetilde{p}=p$.

  For the transform formula, $\mu(\lbrace t\rbrace\times E\times\Rp)=0$ almost surely, since the corresponding $\nu$-mass vanishes identically (cf.\ Step~1 of the proof of Theorem~\ref{thm:spatial_transform}). On this event, the exponential form of the kernel and the flow property of $\phi$ give, for all $(r,x)\in(t,T]\times E$,
  \[
  X_t(r,x) = e^{-\varrho(r-t)}\left(\phi_t(x)+\int_{(0,t]}e^{-\varrho(t-u)}\kappa(x,w)\,z\,d\mu\right) = e^{-\varrho(r-t)}\,V_t(x),
  \]
  where the last equality uses the pathwise identity \eqref{eq:pathwise_identity}. Substituting this into \eqref{eq:conditional_transform} and using $\psi_g(r,y)=(\kappa^*p_r)(y)$, the exponent becomes
  \[
  \begin{aligned}
  -\int_t^T\bigl\langle e^{-\varrho(r-t)}g_r,V_t\bigr\rangle\,dr - \lambda_0\int_t^T\int_EF\bigl((\kappa^*p_r)(y)\bigr)m(dy)\,dr \qquad& \\
  - \lambda_1\int_t^T\int_Ee^{-\varrho(r-t)}V_t(y)\,F\bigl((\kappa^*p_r)(y)\bigr)m(dy)\,dr&,
  \end{aligned}
  \]
  and collecting the two $V_t$-terms via \eqref{eq:p_mild} yields $-A_t-\langle p_t,V_t\rangle$.
\end{proof}

\section{Examples}\label{sec:applications}

\subsection{Multivariate Hawkes processes with Volterra kernels}\label{subsec:hawkes}

Let $E=\lbrace1,\ldots,d\rbrace$ be finite and $m$ the counting measure, and write $K_{ij}(r):=K(r,i,j)$ and $V_t(i)$, $\phi_t(i)$ for the components of the field. Admissibility reduces to $K_{ij}\in L^1((0,T])$ for all $i,j$, and Assumption~\ref{ass:K_transform} then holds automatically (Remark~\ref{rem:K_transform_scope}). The stochastic Volterra equation \eqref{eq:strong_solution} becomes the $d$-dimensional system
\[
V_t(i) = \phi_t(i) + \sum_{j=1}^d\int_{(0,t)\times\Rp}K_{ij}(t-s)\,z\,\mu_j(ds,dz), \qquad i=1,\ldots,d,
\]
where $\mu_j:=\mu(\cdot\times\lbrace j\rbrace\times\cdot)$ has compensator $(\lambda_0+\lambda_1V_s(j))\,ds\,\ell(dz)$. Finite total mark mass implies that $\mu_j$ is locally finite. If $\ell((0,\infty))=\infty$, Corollary~\ref{cor:activity_dichotomy} (applied with $B=\lbrace j\rbrace$) ensures that, almost surely, $\mu_j$ has infinitely many atoms in $I$ when $\int_I(\lambda_0+\lambda_1V_s(j))\,ds>0$, and none when the integrated activity vanishes. Thus the system is a multivariate marked Hawkes-type model with possibly singular Volterra kernels. A mark-$z$ jump at node $j$ raises the future state at node $i$ by $K_{ij}(\cdot)z$, and the activity density at node $j$ is affine in $V(j)$.

For the choice $\ell=\delta_1$ (unit marks), $\lambda_0=0$ and $\lambda_1=1$, the accepted-jump measure has finite activity, and the compensator of the counting process $N_j:=\mu_j(\cdot\times\Rp)$ is $V_s(j)\,ds$, so the system reads
\[
V_t(i) = \phi_t(i) + \sum_{j=1}^d\int_{(0,t)}K_{ij}(t-s)\,dN_j(s),
\]
which identifies $V$ as (the predictable version of) the intensity process of a classical multivariate Hawkes process with excitation kernels $K_{ij}$ and baseline $\phi$. General $\ell$, $\lambda_0$ and $\lambda_1$ interpolate between the purely self-excited case and the L\'evy-driven shot-noise case $\lambda_1=0$. The transform formulas are given by the $d$-dimensional Riccati--Volterra system \eqref{eq:riccati_volterra}.
  
\begin{example}[Infinite activity]\label{ex:gamma_levy_measure}
  Consider the gamma-type L\'evy measure $\ell(dz)=a\,z^{-1}e^{-bz}\,dz$ with $a,b>0$. It has infinite total mass and finite first moment $\overline z=a/b$. On intervals of positive integrated activity, the accepted measure is therefore an infinite-activity Hawkes random measure with almost surely infinitely many atoms (Corollary~\ref{cor:activity_dichotomy}). Its L\'evy exponent is
  \[
    F(q) = a\,\log\left(1+\frac{q}{b}\right), \qquad q\geq0,
  \]
  and Theorem~\ref{thm:spatial_transform} yields the joint Laplace transform of $(V_t(1),\ldots,V_t(d))$, while Theorem~\ref{thm:spacetime_transform} yields the Laplace functional of time-integrated node loads. In the scalar case $d=1$ the model is the affine Volterra jump process of Example~\ref{ex:transform_singleton} and, when $K\equiv1$ and $\phi\equiv v_0$, the closed-interval version is a CBI process with branching and immigration mechanisms $-\lambda_1F$ and $\lambda_0F$ in the sign convention of that example.
\end{example}

\begin{example}[Affine self-exciting portfolio loss models]\label{ex:portfolio_credit_risk}
  Affine point processes have been used as top-down models of clustered portfolio losses and default events in, e.g., \cite{ErraisGieseckeGoldberg2010}. In the present interpretation, a node represents a repeatable event category, and $V_t(i)$ is its activity density. The transform formulas below apply directly to activity and integrated activity, but joint transforms involving the accepted-jump measure are not derived here.

  For a two-category example, take the gamma L\'evy measure $\ell$ of Example~\ref{ex:gamma_levy_measure} with $a,b>0$, let $K_{ij}(r)=c_{ij}e^{-\varrho r}$ with $c_{ij}\geq0$ and $\varrho>0$, and take $f=(f_1,f_2)\in\Rp^2$. Writing $\psi_j(r):=\psi_f(r,j)$, the Riccati--Volterra system is equivalent to
  \begin{equation}\label{eq:two_node_riccati}
  \begin{aligned}
  \psi_j'(r)
  &=-\varrho\psi_j(r)
  +\lambda_1a\sum_{i=1}^2c_{ij}\log\left(1+\frac{\psi_i(r)}{b}\right),\\
  \psi_j(0+)&=\sum_{i=1}^2f_i c_{ij},
  \qquad j=1,2,
  \end{aligned}
  \end{equation}
  and the joint transform is therefore
  \begin{equation}\label{eq:two_node_transform}
  \E\left[e^{-f_1V_t(1)-f_2V_t(2)}\right]
  =\exp\left(
  -\sum_{i=1}^2f_i\phi_t(i)
  -a\int_0^t\sum_{j=1}^2
  (\lambda_0+\lambda_1\phi_s(j))
  \log\left(1+\frac{\psi_j(t-s)}{b}\right)ds
  \right).
  \end{equation}
  The off-diagonal coefficients $c_{12}$ and $c_{21}$ quantify cross-excitation: a loss event in one category raises the future activity of the other. Because the gamma L\'evy measure has infinite mass, this specification represents aggregate loss activity.
\end{example}

\subsection{Spatial Hawkes fields with singular response}\label{subsec:rough_hawkes}
Let $(E,d_E)$ be a compact metric space with a finite measure $m$ and let $p_r(x,y)$ be a symmetric Markov transition density with respect to $m$. Consider the kernel
\[
K(r,x,y) = r^{\alpha-1}\,p_r(x,y), \qquad \alpha\in(0,1),
\]
which is admissible by Example~\ref{ex:admissible_kernels} and satisfies Assumption~\ref{ass:K_transform} with $\widetilde{k}=k$ by symmetry (Remark~\ref{rem:K_transform_scope}). An event at location $y$ excites its surroundings with the diffusive profile $p_r(\cdot,y)$ and the time-singular amplitude $r^{\alpha-1}$. The response is integrable in time but unbounded as $r\downarrow0$, so after a jump the field may have an unbounded right-hand response. Both transform formulas apply because the dual majorant $\widetilde{k}(r)=r^{\alpha-1}$ remains integrable. This yields a concrete spatial Hawkes random field with a singular temporal response and diffusive spatial propagation. 

\begin{example}[Spatial-average transform for the diffusive singular field]\label{ex:spatial_average_heat}
Assume in addition that $p_r$ is conservative and symmetric, and take the constant test function $f\equiv\theta$ with $\theta\geq0$. Then
\[
K_f(r,y)=\theta r^{\alpha-1}\int_Ep_r(x,y)m(dx)=\theta r^{\alpha-1}.
\]
The monotone iteration in Lemma~\ref{lem:riccati_volterra} preserves spatially constant functions, and hence the minimal Riccati--Volterra solution is spatially constant, $\psi_f(r,y)=q_\theta(r)$, where
\begin{equation}\label{eq:scalar_fractional_riccati}
q_\theta(r)=\theta r^{\alpha-1}+\lambda_1\int_0^r(r-s)^{\alpha-1}F\bigl(q_\theta(s)\bigr)\,ds,
\qquad r\in(0,T].
\end{equation}
Consequently, the Laplace transform of the spatially aggregated field has the one-dimensional representation
\[
\begin{aligned}
\E\left[\exp\left(-\theta\int_EV_t(x)m(dx)\right)\right]
=\exp\biggl(&-\theta\int_E\phi_t(x)m(dx)\\
&-\int_0^tF\bigl(q_\theta(t-s)\bigr)
\left[\lambda_0m(E)+\lambda_1\int_E\phi_s(y)m(dy)\right]ds\biggr).
\end{aligned}
\]
Thus a genuinely continuum-space model with diffusive propagation reduces, for spatial-average claims, to a scalar weakly singular Riccati--Volterra equation. The Riccati step therefore requires no spatial discretization and only the deterministic spatial integrals of the driver $\phi$ remain to be evaluated.
\end{example}

In what follows, we consider the spatial regularity of the Hawkes type field, which relies on the results of Section~\ref{sec:regularity}.  
\begin{example}[Separable singular kernels with H\"older modifications]\label{ex:separable_holder}
  Fix $p\geq2$. Let $(E,d_E)$ be compact with $m(E)<\infty$ and covering numbers $N(E,\varepsilon)\leq c\,\varepsilon^{-D}$ for $\varepsilon\in(0,1]$, and let
  \[
  K(r,x,y) = r^{\alpha-1}\,G(x,y), \qquad \alpha\in(0,1),
  \]
  with a bounded Borel function $G:E\times E\to\Rp$ satisfying, for some $\beta\in(0,1]$ and constants $C_q\geq0$,
  \[
  \left\|G(x,\cdot)-G(x',\cdot)\right\|_{L^q(m)} \leq C_q\,d_E(x,x')^\beta, \qquad x,x'\in E,\ q\in\lbrace1,2,p\rbrace.
  \]
  Assume that $\phi$ satisfies \eqref{eq:phi_modulus}, that $\int_0^\infty z^p\ell(dz)<\infty$, and that
  \[
  \alpha>1-\frac1p \qquad\text{and}\qquad \beta p>D.
  \]
  Then $k_q(r)\leq \lVert G\rVert_\infty^q\,m(E)\,r^{-q(1-\alpha)}$ and \eqref{eq:K_Lq_modulus} holds with $h_q(r)=C_q\,r^{\alpha-1}$. Since $q(1-\alpha)\leq p(1-\alpha)<1$ for $q\in\lbrace1,2,p\rbrace$, we obtain $k_2,k_p\in L^1((0,T])$ and $h_q\in L^q((0,T])$ for all three exponents, so Propositions~\ref{prop:p_moments} and~\ref{prop:increment_moments} and Theorem~\ref{thm:holder_modification} apply. Consequently, for every $t\in[0,T]$, the spatial field $x\mapsto V_t(x)$ admits a modification that is H\"older continuous of every order $\gamma<\beta-D/p$. A concrete admissible tuple is $D=1$, $p=4$, $\beta=1$, $\alpha=4/5$: then $p(1-\alpha)=4/5<1$ and $\beta p=4>1=D$, and, provided $\overline{z}_4<\infty$ and $\phi$ is spatially Lipschitz uniformly in time, the field admits fixed-time spatial modifications of every H\"older order below $3/4$. Finally, $\widetilde{k}(r)=r^{\alpha-1}\sup_y\int_EG(x,y)\,m(dx)\leq \lVert G\rVert_\infty m(E)\,r^{\alpha-1}$, so Assumption~\ref{ass:K_transform} holds automatically for bounded $G$ and finite $m$, and the transform formulas of Sections~\ref{sec:spatial_transform} and~\ref{sec:spacetime_transform} apply to the same kernel.
\end{example}

Example~\ref{ex:separable_holder} constructs a field with a time-singular kernel, but spatially H\"older continuous paths, as a consequence of the spatial increment bounds on $G$, the temporal integrability conditions, and the assumed spatial regularity of $\phi$. The situation is different for a time-diagonal heat kernel, whose spatial concentration becomes stronger as the lag tends to zero in which case we are not able to obtain a spatially H\"older continuous modification via the results of Section~\ref{sec:regularity}. For example, consider the flat torus $E=\mathbb{T}^d=\mathbb{R}^d/\mathbb{Z}^d$ with its periodized Gaussian heat kernel $p_r$, represented as
\[
p_r(x,y)=\sum_{k\in\mathbb Z^d}(4\pi r)^{-d/2}
\exp\left(-\frac{|x-y+k|^2}{4r}\right), \qquad x,y\in \mathbb{R}^d.
\]
In order to assess spatial regularity via the results of Section~\ref{sec:regularity}, we require bounds on the spatial increments of the field, which are controlled by spatial increments of the kernel. Note that for $0<r\leq 1$, termwise Gaussian integration yields the estimate $\lVert p_r(x,\cdot)\rVert_{L^q}^q\asymp r^{-d(q-1)/2}$. Hence we find that
\[
k_q(r) \asymp r^{-q(1-\alpha)-d(q-1)/2}\quad(r\downarrow0), \qquad\text{so that}\qquad k_q\in L^1((0,T])\iff q(1-\alpha)+\frac{d(q-1)}{2}<1.
\]
Since the exponent is increasing in $q$, the single condition $p(1-\alpha)+d(p-1)/2<1$ ensures $k_2,k_p\in L^1$, and Proposition~\ref{prop:p_moments} yields uniform $p$-th moment bounds whenever $\int_0^\infty z^p\ell(dz)<\infty$; for example, for $d=1$ and $p=2$ this requires $\alpha>3/4$. Note then that the following bounds also hold
\begin{align*}
\left\|p_r(x,\cdot)-p_r(x',\cdot)\right\|_{L^q} &\leq 2\sup_{\xi}\left\|p_r(\xi,\cdot)\right\|_{L^q},\\
\left\|p_r(x,\cdot)-p_r(x',\cdot)\right\|_{L^q} &\leq d_E(x,x')\,\sup_{\xi}\left\|\nabla_xp_r(\xi,\cdot)\right\|_{L^q},
\end{align*}
where the second follows from the fundamental theorem of calculus along a torus geodesic. Differentiating the periodized Gaussian term by term and integrating the Gaussian derivative gives $\sup_\xi\lVert\nabla_xp_r(\xi,\cdot)\rVert_{L^q}\lesssim r^{-1/2-d(1-1/q)/2}$, while the preceding size estimate gives $\sup_\xi\lVert p_r(\xi,\cdot)\rVert_{L^q}\lesssim r^{-d(1-1/q)/2}$; see also \cite[Chapters~2--3]{Davies1989}. Interpolation yields, for $q\geq1$ and $\beta\in(0,1]$,
\begin{equation}\label{eq:heat_kernel_Lq_bound}
\left\|p_r(x,\cdot)-p_r(x',\cdot)\right\|_{L^q}\leq C_{q,\beta}\,d_E(x,x')^\beta r^{-\beta/2-d(1-1/q)/2},\qquad 0<r\leq T.
\end{equation}
Thus \eqref{eq:K_Lq_modulus} holds with
$h_q(r)=C_{q,\beta}r^{\alpha-1-\beta/2-d(1-1/q)/2}$, and $h_q\in L^q((0,T])$ precisely when
\[
q\left(1-\alpha+\frac{\beta}{2}\right)+\frac{d(q-1)}{2}<1.
\]
This requirement is, however, incompatible with the covering condition $\beta p>d$ of Theorem~\ref{thm:holder_modification}: since $\alpha<1$, the condition $\beta p>d$ forces
\[
p\left(1-\alpha+\frac{\beta}{2}\right)+\frac{d(p-1)}{2}>\frac{p\beta}{2}+\frac{d(p-1)}{2}>\frac{d}{2}+\frac{d(p-1)}{2}=\frac{dp}{2}\geq1
\]
for $p\geq2$ and $d\geq1$, so that $h_p\notin L^p((0,T])$ whenever $\beta p>d$. The Kolmogorov-Chentsov type criterion of Theorem~\ref{thm:holder_modification} therefore cannot yield a spatially H\"older field modification for the time-diagonal heat kernel $r^{\alpha-1}p_r$, since the temporal singularity and spatial concentration leave no admissible exponent range. For $q=1$, the bound \eqref{eq:heat_kernel_Lq_bound} gives an integrable modulus $h_1(r)=Cr^{\alpha-1-\beta/2}$ whenever $0<\beta\leq1$ and $\beta<2\alpha$. If $\phi$ satisfies \eqref{eq:phi_modulus} with this exponent, Theorem~\ref{thm:holder_in_mean} therefore yields spatial $\beta$-H\"older continuity in mean. The first-moment construction and transform formulas remain applicable.

\subsection{Frozen forward-state representation}\label{subsec:forward_curves}
Consider the case where $V_t(x)$ is interpreted as the time-$t$ level of a non-negative random surface indexed by $x\in E$. This includes the portfolio setting in Example~\ref{ex:portfolio_credit_risk}, with $E$ representing a continuum of interconnected markets or risk classifications. In that case, $V_t(x)$ represents an activity density. The frozen field $X_t(r,x)$ in \eqref{eq:forward_curve_lift} records the driver and the jumps observed by time $t$ and is the \emph{input} to the conditional-mean equation in Lemma~\ref{lem:conditional_mean}. Indeed, the conditional mean $h_t(r,x)=\E[V_r(x)\mid\Fcal_t^N]$ solves \eqref{eq:conditional_mean_volterra} and due to \eqref{eq:h_t_neumann}, it is a measurable functional of $X_t$.

Let $g$ be a bounded non-negative function on $[0,T]\times E$. The conditional Laplace transform
\[
\E\left[\exp\left(-\int_t^T\int_Eg_s(x)\,V_s(x)\,m(dx)\,ds\right)\Bigm\vert\Fcal_t^N\right],
\]
is given by Theorem~\ref{thm:conditional_transform}. For exponential kernels, under the driver flow assumption of Corollary~\ref{cor:exponential_conditional}, the transform is exponentially affine in $V_t$ whereas for general kernels it depends on the full frozen field $X_t$. Such forward-state representations underlie Markovian lifts in affine Volterra theory \cite{AbiJaberLarssonPulido2019,BondiLivieriPulido2024,CuchieroTeichmann2020} and are widely used for pricing problems in finance. A pricing interpretation, however, requires an additional pricing measure and asset model, which we do not pursue here.

\subsection{A predictable branching type representation}
  Let $E$ be compact, let $m$ be finite, and let $(P_r)_{r\geq0}$ be a conservative symmetric Markov semigroup with transition density $p_r(x,y)$ with respect to $m$. Suppose that $(r,x,y)\mapsto p_r(x,y)$ is Borel measurable for $r>0$ and set $K(r,x,y)=p_r(x,y)$. Conservativity and symmetry give
  \[
  \sup_{r\in(0,T],x\in E}\int_Ep_r(x,y)m(dy)=1,
  \qquad
  \sup_{r\in(0,T],y\in E}\int_Ep_r(x,y)m(dx)=1,
  \]
  so the row and dual-column assumptions hold. Define
  \[
  (P_rf)(x):=\int_Ep_r(x,y)f(y)m(dy),
  \qquad \langle P_r^\ast M,f\rangle:=\langle M,P_rf\rangle.
  \]
  For $V_0\in B_b(E)_+$, set $M_0(dx)=V_0(x)m(dx)$, $\phi_0=V_0$, and
  \[
  \phi_t(x)=P_tV_0(x)=\int_Ep_t(x,y)V_0(y)m(dy),\qquad t>0.
  \]
  Then $\phi$ is bounded and non-negative and the predictable random measure
  \[
  M_t^{\mathrm{pred}}(dx):=V_t(x)m(dx),
  \]
  is finite for each fixed $t$, almost surely. Under our strict-past convention, \eqref{eq:pathwise_identity} is equivalent to
  \[
  M_t^{\mathrm{pred}}=P_t^\ast M_0+
  \int_{(0,t)\times E\times\Rp}zP_{t-s}^\ast\delta_y\,\mu(ds,dy,dz),
  \]
  where the accepted-jump measure has compensator
  \[
  \nu(ds,dy,dz)=\bigl(\lambda_0m(dy)+\lambda_1M_s^{\mathrm{pred}}(dy)\bigr)ds\,\ell(dz).
  \]
  This accepted-jump measure identity concerns the predictable representative used in the thinning equation. A measure-valued c\`adl\`ag realization, when available, is instead naturally written with a closed-interval stochastic convolution and contains the atom created at the current time. The two conventions agree at every deterministic time almost surely, because the driving Poisson random measure has diffuse time intensity, but they need not agree at random jump times. We therefore do not identify $M^{\mathrm{pred}}$ with a c\`adl\`ag density-valued branching process. The displayed compensator and mild identity nevertheless exhibit the branching-with-immigration mechanism at the predictable level.

  This connection also appears at the transform level. Indeed, for $f\in B_b(E)_+$, the Riccati-Volterra equation becomes the mild log-Laplace equation
  \[
  u_t=P_tf+\lambda_1\int_0^tP_{t-s}\bigl[F(u_s)\bigr]ds,
  \qquad u_t(y):=\psi_f(t,y)\quad(t>0),
  \]
  with $u_0=f$, distinct from the auxiliary convention $\psi_f(0,\cdot)=0$ used in stochastic integrals, and Theorem~\ref{thm:spatial_transform} reduces to
  \[
  \E\left[e^{-\langle f,M_t^{\mathrm{pred}}\rangle}\right]
  =\exp\left(-\langle u_t,M_0\rangle
  -\lambda_0\int_0^t\langle F(u_s),m\rangle ds\right),
  \]
  which has the standard branching-with-immigration log-Laplace form (compare with \cite{DawsonLi2012}).

\section{Spatial moments and regularity}\label{sec:regularity}
We now give sufficient conditions for spatial regularity of $V_t(x)$. Fix a metric $d_E$ compatible with the topology of $E$. These results are independent of the transform theory. The criteria apply in particular to singular kernels with non-concentrating spatial factors, as illustrated in Section~\ref{subsec:rough_hawkes}; for time-diagonal heat-type kernels, Theorem~\ref{thm:holder_modification} is only a sufficient criterion and need not be sharp. Proposition~\ref{prop:increment_moments} in Appendix~\ref{app:spatial_regularity} gives the corresponding higher-moment increment estimate used for H\"older modifications.

In order to derive estimates on spatial increments, we require higher order mark moments and for $q\geq 1$ we therefore introduce the notation
\[
\overline{z}_q := \int_0^\infty z^q\,\ell(dz) \qquad\text{and}\qquad k_q(r):=\sup_{x\in E}\int_EK(r,x,y)^q\,m(dy),
\]
so that $\overline{z}_1 = \overline{z}$ and $k_1=k$, and we note that $\overline{z}<\infty$ together with $\overline{z}_p<\infty$ implies $\overline{z}_q<\infty$ for all $q\in[1,p]$, since $z^q\leq z+z^p$.

\begin{theorem}\label{thm:holder_in_mean}
  Assume that there exist $\beta\in(0,1]$, a non-negative $h\in L^1([0,T])$ and a constant $C_\phi\geq 0$ such that, for all $x,x'\in E$,
  \begin{align}
  \int_E \lvert K(r,x,y)-K(r,x',y)\rvert\,m(dy) &\leq h(r)\,d_E(x,x')^\beta, \qquad r\in(0,T],\label{eq:K_L1_modulus}\\
  \sup_{t\in[0,T]}\,\lvert \phi_t(x)-\phi_t(x')\rvert &\leq C_\phi\,d_E(x,x')^\beta.\label{eq:phi_modulus}
  \end{align}
  Then, with $M:=\sup_{(t,x)\in[0,T]\times E}\E[V_t(x)]<\infty$ as in \eqref{eq:V_sup_moment_bound}, it holds that
  \[
  \sup_{t\in[0,T]}\E\left[\lvert V_t(x)-V_t(x')\rvert\right] \leq \left(C_\phi+\overline{z}\left(\lambda_0+\lambda_1M\right)\int_0^Th(r)\,dr\right)d_E(x,x')^\beta, \qquad x,x'\in E.
  \]
\end{theorem}
\begin{proof}
  Fix $t\in[0,T]$ and $x,x'\in E$ and subtract the two instances of \eqref{eq:strong_solution}. Since both stochastic integrals are almost surely finite (see the proof of Lemma~\ref{lem:L1_estimate}), we obtain the pathwise bound
  \[
  \begin{aligned}
  \lvert V_t(x)-V_t(x')\rvert &\leq \lvert \phi_t(x)-\phi_t(x')\rvert \\
  &+ \int_{(0,t)\times E\times\Rp\times\Rp}
  \lvert K(t-s,x,y)-K(t-s,x',y)\rvert\,z\\
  &\hspace{5em}\times\mathbf{1}_{\lbrace u\leq\lambda_0+\lambda_1V_s(y)\rbrace}(u)
  \,N(ds,dy,dz,du).
  \end{aligned}
  \]
  Taking expectations and applying formula \eqref{eq:campbell} yields
  \[
  \begin{aligned}
  \E\left[\lvert V_t(x)-V_t(x')\rvert\right]
  \leq{}& C_\phi\,d_E(x,x')^\beta\\
  &+\overline{z}\int_0^t\int_E
  \lvert K(t-s,x,y)-K(t-s,x',y)\rvert
  \left(\lambda_0+\lambda_1\E[V_s(y)]\right)m(dy)\,ds,
  \end{aligned}
  \]
  and the claim follows from \eqref{eq:K_L1_modulus} and $\E[V_s(y)]\leq M$.
\end{proof}

\begin{theorem}\label{thm:holder_modification}
  Let $p\geq 2$ and assume that $\overline{z}_p<\infty$ and that $k_2$ and $k_p$ are Borel measurable with $k_2,k_p\in L^1([0,T])$. Suppose in addition that \eqref{eq:phi_modulus} holds and that there exist a $\beta\in(0,1]$ and non-negative functions $h_q\in L^q([0,T])$, $q\in\lbrace1,2,p\rbrace$, such that for all $x,x'\in E$,
  \[
  \left(\int_E\lvert K(r,x,y)-K(r,x',y)\rvert^q\,m(dy)\right)^{1/q} \leq h_q(r)\,d_E(x,x')^\beta, \qquad r\in(0,T],\ q\in\lbrace1,2,p\rbrace.
  \]
  Finally, suppose that $(E,d_E)$ is compact and that there exist constants $c,D>0$ such that the covering numbers of $E$ satisfy $\mathcal{N}(E,\varepsilon)\leq c\,\varepsilon^{-D}$ for all $\varepsilon\in(0,1]$, where $\mathcal{N}(E,\varepsilon)$ denotes the minimal number of closed $d_E$-balls of radius $\varepsilon$ needed to cover $E$, and that $\beta p>D$. Then, for every $t\in[0,T]$, the spatial field $x\mapsto V_t(x)$ admits a modification which is almost surely $\alpha$-H\"older continuous for every $\alpha\in\left(0,\beta-D/p\right)$ and this modification can be chosen simultaneously for all such $\alpha$.
\end{theorem}
\begin{proof}
  We verify the hypotheses of \cite[Theorem~1.1]{KratschmerUrusov2022} with parameter space $\Theta=E$, state space $\mathcal{X}=\R$, moment order $p$, increment exponent $q:=\beta p$ for the parameter-space distance, and covering exponent $D$ (denoted $t$ in \cite{KratschmerUrusov2022}, but we avoid that letter here, since $t$ is the fixed time of the statement). Since $q>D$ and $E$ is compact, it is totally bounded with finite diameter $\Delta(E)$, and the covering-number condition of \cite[Eq.~(1)]{KratschmerUrusov2022} holds for all $\varepsilon\in(0,\Delta(E)]$ with the constant $c\,(1\vee\Delta(E))^{D}$. For $\varepsilon\le 1$ this is the assumption, and for $\varepsilon\in(1,\Delta(E)]$ we have $\mathcal{N}(E,\varepsilon)\leq \mathcal{N}(E,1)\le c\le c\,(1\vee \Delta(E))^D\varepsilon^{-D}$. The measurability condition \cite[Eq.~(2)]{KratschmerUrusov2022} is automatic since $\R$ is separable and the moment condition \cite[Eq.~(3)]{KratschmerUrusov2022} is the bound $\E[\lvert V_t(x)-V_t(x')\rvert^p]\leq C\,d_E(x,x')^{\beta p}$ of Proposition~\ref{prop:increment_moments}. Since $\R$ is complete, \cite[Theorem~1.1]{KratschmerUrusov2022} yields an $\alpha$-H\"older modification for each $\alpha\in(0,(q-D)/p)$. To obtain one modification for all such exponents, choose $\alpha_n\uparrow(q-D)/p$ and corresponding continuous modifications $V^{(n)}_t$. On a fixed countable dense subset of $E$, all $V^{(n)}_t$ agree almost surely with the original field and hence with each other. Intersecting this event with the countably many events on which $V^{(n)}_t$ is $\alpha_n$-H\"older still gives a common event of probability one and continuity makes the modifications identical on all of $E$ there. Any one of these versions, completed on the null complement with a fixed continuous function, therefore has the $\alpha_n$-H\"older property for every $n$ and is $\alpha$-H\"older for every $\alpha<(q-D)/p=\beta-D/p$.
\end{proof}

\section*{Disclosure statement}
The authors declare no conflicts of interest.

\section*{Funding}
T. K. Kloster gratefully acknowledges financial support from the Center of Research in Energy: Economics and Markets and The Danish Council of Independent Research under DFF grant 10.46540/5247-00005B.

Sven Karbach gratefully acknowledges support by the University of Amsterdam’s
interfaculty Research Priority Area \emph{Energy Transition through
the Lens of Sustainable Development Goals} (ENLENS).

\appendix

\section{Measurability lemmas}\label{app:measurability}
\begin{proof}[Proof of Lemma~\ref{lem:distinct_times}]
  The intensity of $\lbrace0\rbrace\times E\times\Rp\times\Rp$ is zero, so $N$ does not charge time zero almost surely. The disjoint mark bands
  \[
  I_a^-=(1/(a+1),1/a],\qquad I_a^+=(a,a+1],\qquad a\in\N,
  \]
  exhaust $(0,\infty)$, and the intervals $J_b=[b,b+1)$, $b\in\N_0$, partition $\Rp$. The products $I_a^\pm\times J_b$ therefore form a countable, disjoint partition of $(0,\infty)\times\Rp$ into sets of finite $\ell\otimes du$-measure. We enumerate these sets as $(Q_i)_{i\in\N}$. The restriction $N_i$ of $N$ to $[0,T]\times E\times Q_i$ has finite total intensity $\Lambda_i:=T\,m(E)\,(\ell\otimes du)(Q_i)$, and restrictions to distinct cells are independent.

  Fix $i,j\in\N$ and $K\in\N$, and partition $(0,T]$ into $K$ intervals of length $T/K$. If an atom of $N_i$ and a distinct atom of $N_j$ share a time coordinate, then some interval carries two atoms of $N_i$ when $i=j$, or an atom of each of $N_i$ and $N_j$ when $i\neq j$. The elementary Poisson bounds $\Prob(\mathrm{Poi}(\lambda)\geq1)\leq\lambda$ and $\Prob(\mathrm{Poi}(\lambda)\geq2)\leq\lambda^2$, together with independence for $i\neq j$, bound the probability of this event by $\Lambda_i\Lambda_j/K$ (and by $\Lambda_i^2/K$ when $i=j$). Letting $K\to\infty$ proves that the event has probability zero for each pair $(i,j)$. Taking the countable union over all pairs completes the proof.
\end{proof}

\begin{proof}[Proof of Lemma~\ref{lem:predictability}]
  Fix $n\in\N$ and restrict $N$ to $[0,T]\times E\times(\tfrac1n,n]\times[0,n]$. The resulting random measure $N_n$ has finite intensity $T\,m(E)\,\ell((\tfrac1n,n])\,n$ and is therefore almost surely a finite sum of Dirac measures. Its time marginal is atomless, so its atoms almost surely have distinct time coordinates.

  Because $[0,T]\times E\times(\tfrac1n,n]\times[0,n]$ is a standard Borel space, finite point measures on it admit a measurable enumeration and we may order the atoms of $N_n$ as $(\sigma_i,Y_i,Z_i,U_i)_{i\in\N}$. The finite-valued times are strictly increasing almost surely, since jumps happen at distinct time coordinates almost surely cf. Lemma~\ref{lem:distinct_times}. If fewer than $i$ atoms occur, set $\sigma_i=\infty$ and $(Y_i,Z_i,U_i)=(y_0,0,0)$ for a fixed $y_0\in E$. Each $\sigma_i$ is then a stopping time since
  \[
  \{\sigma_i\leq t\}=\{N_n((0,t]\times E\times(\tfrac1n,n]\times[0,n])\geq i\}\in\Fcal_t^N.
  \]
  For every Borel set $A\subseteq E\times(\tfrac1n,n]\times[0,n]$, the event
  $
  \{\sigma_i\leq t,\ (Y_i,Z_i,U_i)\in A\}
  $
  is a measurable function of the restriction of $N_n$ to $(0,t]\times E\times(\tfrac1n,n]\times[0,n]$ and hence belongs to $\Fcal_t^N$. By the definition of the stopping-time $\sigma$-algebra, $(Y_i,Z_i,U_i)$ is therefore $\Fcal_{\sigma_i}^N$-measurable.

  Next, define the random variable $W_i$ by setting it equal to $0$ on $\{\sigma_i=\infty\}$ and to $W(\sigma_i,Y_i,Z_i,U_i)$ on $\{\sigma_i<\infty\}$. We show that $W_i$ is $\Fcal_{\sigma_i}^N$-measurable. For $W$ of the product form
  \[
  W(\omega,s,y,z,u) = Y(\omega)\mathbf{1}_{(a,b]}(s)\mathbf{1}_{A}(y)\mathbf{1}_{B}(z)\mathbf{1}_{C}(u),
  \]
  with $Y$ bounded and $\Fcal_a^N$-measurable, the claim follows from the $\Fcal_{\sigma_i}^N$-measurability of $Y\mathbf{1}_{\{a<\sigma_i\}}$ and of the marks. Predictable generators supported on $\lbrace0\rbrace$ contribute nothing because $N_n$ has no atom at time zero almost surely. A functional monotone-class argument, first for indicators and then for bounded non-negative functions, extends the claim to every $\Pcal\otimes\Bcal(\Rp)\otimes\Bcal(\Rp)$-measurable $W$ with values in $[0,1]$.

  Now define the random field
  \[
  \begin{aligned}
  \Psi^n_t(x) &:= \int_{(0,t)\times E\times \Rp\times \Rp}K(t-s,x,y)\,z\,W(s,y,z,u)\,N_n(ds,dy,dz,du) \\
  &\phantom{:}= \sum_{i:\,\sigma_i<t}K(t-\sigma_i,x,Y_i)\,Z_i\,W_i,
  \end{aligned}
  \]
  and fix $i\in\N$. The set $D_i := \lbrace (\omega,t,x): \sigma_i(\omega)<t\rbrace$ is predictable, and on $D_i$ the three mappings
  \begin{enumerate}[i]
    \item $(\omega,t,x)\mapsto t-\sigma_i(\omega)$, 
    \item $(\omega,t,x)\mapsto x$,
    \item $(\omega,t,x)\mapsto Y_i(\omega)$,
  \end{enumerate}
   are $\Pcal$-measurable. (i) follows since the map extended by $0$ on the complement of $D_i$ is left-continuous in $t$ and adapted. (ii) follows since it is measurable with respect to $\Pcal_{[0,T]}\otimes \Ecal$, with $\Pcal_{[0,T]}$ being the restriction of $\Pcal$ to the time interval $[0,T]$. (iii) follows since, for every $A\in\Ecal$, the field $\mathbf{1}_{\{Y_i\in A\}}\mathbf{1}_{\{\sigma_i<t\}}$ is of the form $Z\mathbf{1}_{(\sigma_i,T]}(t)$ with $Z$ $\Fcal_{\sigma_i}^N$-measurable, hence predictable. Since $E$ is Polish, the Borel $\sigma$-algebra of $(0,T]\times E\times E$ is the product of the Borel $\sigma$-algebras, so that the mapping 
   \[
   \Theta_i:=(t-\sigma_i,\,x,\,Y_i): D_i\to (0,T]\times E\times E,
   \]
   is measurable ($t-\sigma_i\in(0,T]$ on $D_i$). Thus, $\mathbf{1}_{D_i}\,(K\circ \Theta_i)$ is a predictable random field, and multiplying it with the predictable field $Z_iW_i\mathbf{1}_{(\sigma_i,T]}(t)$ shows that each summand of $\Psi^n$ -- hence $\Psi^n$ itself -- is predictable.

  Finally, letting $n\to\infty$ we have $(\tfrac1n,n]\times[0,n]\uparrow (0,\infty)\times\Rp$, and since the factor $z$ vanishes on $\lbrace z=0\rbrace$, monotone convergence yields $\Psi^n_t(x)\uparrow \Psi_t(x)$ for every $(\omega,t,x)$. Predictability is preserved under pointwise limits, which completes the proof.
\end{proof}

\section{Spatially inhomogeneous affine coefficients}\label{app:inhomogeneous_coefficients}

In the next proposition, we show that the constant coefficients used throughout the paper are not essential for the theory to hold.

\begin{proposition}\label{prop:inhomogeneous_coefficients}
Let $\lambda_0,\lambda_1\in B_b(E)_+$ and replace every thinning indicator in the construction by
\[
\mathbf 1_{\{u\leq \lambda_0(y)+\lambda_1(y)X_s(y)\}}.
\]
Then the existence and uniqueness theorem, the compensator and activity statements, the first- and higher-moment bounds, and the fixed-time, integrated, and conditional transform formulas remain correct with the following modifications.
\begin{enumerate}[(i)]
\item The contraction constant in Theorem~\ref{thm:Psi_contraction} is $\overline z\lVert\lambda_1\rVert_\infty J_\eta(T)$, while the affine bound contains $\lVert\lambda_0\rVert_\infty$. Hence one chooses $\eta$ such that $\overline z\lVert\lambda_1\rVert_\infty J_\eta(T)<1$.
\item The accepted-jump compensator is
\[
\nu(ds,dy,dz)=\bigl(\lambda_0(y)+\lambda_1(y)V_s(y)\bigr)\,ds\,m(dy)\,\ell(dz),
\]
and the first moment solves
\[
 f(t,x)=\phi_t(x)+\overline z\int_0^t\int_EK(t-s,x,y)\bigl(\lambda_0(y)+\lambda_1(y)f(s,y)\bigr)m(dy)\,ds.
\]
For the corresponding resolvent representation define the weighted kernel $K^{(1)}(r,x,y):=K(r,x,y)\lambda_1(y)$ and
\[
g(t,x):=\phi_t(x)+\overline z\int_0^t\int_EK(t-s,x,y)\lambda_0(y)m(dy)\,ds.
\]
Then $R^{(1)}:=\sum_{n\geq1}\overline z^{\,n}(K^{(1)})^{\star n}$ is controlled by the same exponential-weight argument with $\lVert\lambda_1\rVert_\infty k$, and $f=g+R^{(1)}\star g$.
\item Under Assumption~\ref{ass:K_transform}, the fixed-time Riccati--Volterra equation becomes
\[
\psi_f(r,y)=K_f(r,y)+\int_0^r\int_EK(r-s,y',y)\lambda_1(y')F\bigl(\psi_f(s,y')\bigr)m(dy')\,ds.
\]
The monotone solution is bounded by the inhomogeneous dual resolvent obtained from \eqref{eq:dual_resolvent} by replacing $\lambda_1$ with $\lVert\lambda_1\rVert_\infty$. Theorem~\ref{thm:spatial_transform} becomes
\[
\E[e^{-\langle f,V_t\rangle}]
=\exp\left(-\langle f,\phi_t\rangle-\int_0^t\int_E\bigl(\lambda_0(y)+\lambda_1(y)\phi_s(y)\bigr)F\bigl(\psi_f(t-s,y)\bigr)m(dy)\,ds\right).
\]
\item For $g\in B_b([0,T]\times E)_+$, the backward equation becomes
\[
\psi_g(s,y)=G_g(s,y)+\int_s^T\int_EK(r-s,y',y)\lambda_1(y')F\bigl(\psi_g(r,y')\bigr)m(dy')\,dr,
\]
and the integrated and conditional transform formulas are obtained by replacing $\lambda_0+\lambda_1\phi$ and $\lambda_0+\lambda_1X_t$ by $\lambda_0(\cdot)+\lambda_1(\cdot)\phi$ and $\lambda_0(\cdot)+\lambda_1(\cdot)X_t$, respectively.
\end{enumerate}
The same substitutions apply to the moment estimates, with constants depending on $\lVert\lambda_0\rVert_\infty$ and $\lVert\lambda_1\rVert_\infty$.
\end{proposition}
\begin{proof}
For the coupling estimate, integrate the difference of the two thinning indicators over the auxiliary mark $u$ and use
\[
\int_0^\infty\left|\mathbf 1_{\{u\leq\lambda_0(y)+\lambda_1(y)a\}}-\mathbf 1_{\{u\leq\lambda_0(y)+\lambda_1(y)b\}}\right|du
=\lambda_1(y)|a-b|
\leq\lVert\lambda_1\rVert_\infty|a-b|.
\]
All remaining existence and moment estimates then follow with sup-norm bounds on the coefficients. The compensator identity follows by integrating the thinning indicator over $u$. In the transform proofs, the term proportional to the propagated field at location $y'$ carries the multiplier $\lambda_1(y')$. Inserting this multiplier into the dual Riccati equations gives exactly the cancellations used in Theorems~\ref{thm:spatial_transform} and~\ref{thm:spacetime_transform}. The remaining arguments are unchanged after these substitutions.
\end{proof}

\section{Proof of Theorem~\ref{thm:strong_existence} and Corollary~\ref{cor:activity_dichotomy}}
\begin{lemma}\label{lem:H_eta_Banach}
  For every $\eta\geq 0$, the pair $(\mathcal{H}_\eta,\lVert\cdot\rVert_\eta)$ is a Banach space and $\mathcal{H}_\eta^+$ is a closed subset of $\mathcal{H}_\eta$.
\end{lemma}
\begin{proof}
  Only completeness requires an argument. Let $(X^n)_{n\in\N}$ be a Cauchy sequence in $\mathcal{H}_\eta$ and fix predictable representatives. We may pass to a subsequence with $\lVert X^{n_{k+1}}-X^{n_k}\rVert_\eta\leq 2^{-k}$. Then there exists a $dt\otimes m$-null set outside of which $\sum_{k}\E[\lvert X^{n_{k+1}}_t(x)-X^{n_k}_t(x)\rvert]\leq e^{\eta T}\sum_k 2^{-k}<\infty$, and hence $(X^{n_k}_t(x))_k$ converges almost surely and in $L^1(\Prob)$. The field $X:=\limsup_k X^{n_k}\,\mathbf{1}_{\{\lvert \limsup_k X^{n_k}\rvert<\infty\}}$ is predictable, and by Fatou's lemma it follows that
  \[
  e^{-\eta t}\E[\lvert X_t(x)-X^{n_k}_t(x)\rvert]\leq \sum_{j\geq k}2^{-j},
  \]
  for $dt\otimes m$-a.e. $(t,x)$, so that $X^{n_k}\to X$ in $\mathcal{H}_\eta$. By the Cauchy property, it therefore holds that $X^n\to X$. To show that $\mathcal{H}_\eta^+$ is closed, convergence in $\mathcal{H}_\eta$ implies convergence in $L^1(\Prob\otimes e^{-\eta t}dt\otimes m)$ which yields $\Prob\otimes dt\otimes m$-a.e. convergence along a subsequence, which preserves non-negativity.
\end{proof}

\paragraph{Proof of Theorem~\ref{thm:strong_existence}}
\begin{proof}
  Fix $\eta>0$ with $\overline{z}\lambda_1 J_\eta(T)<1$, let $Y\in\mathcal{H}_\eta^+$ be the fixed point from Theorem~\ref{thm:Psi_contraction}, and let $V=\Psi(Y)$ be the PRM-strong predictable solution furnished by Lemma~\ref{lem:pointwise_reconstruction}. For the moment bound, from \eqref{eq:Psi_bound} and $\Psi(Y)=Y$ in $\mathcal{H}_\eta$ we get $\lVert Y\rVert_\eta \leq \lVert\phi\rVert_\infty + \overline{z}\left(\lambda_0+\lambda_1\lVert Y\rVert_\eta\right)J_\eta(T)$, and hence
  \[
  \lVert Y\rVert_\eta \leq \frac{\lVert\phi\rVert_\infty + \overline{z}\lambda_0 J_\eta (T)}{1-\overline{z}\lambda_1 J_\eta (T)} < \infty.
  \]
  Then, by equation \eqref{eq:campbell}, for every $(t,x)\in [0,T]\times E$,
  \[
  \begin{aligned}
  \E\left[V_t(x)\right] &= \phi_t(x) + \overline{z}\int_0^t\int_E K(t-s,x,y)\left(\lambda_0+\lambda_1\E[Y_s(y)]\right)m(dy)\,ds \\
  &\leq \lVert\phi\rVert_\infty+\overline{z}\left(\lambda_0+\lambda_1 e^{\eta T}\lVert Y\rVert_\eta\right)\int_0^T k(r)\,dr,
  \end{aligned}
  \]
  where we have used that $\E[Y_s(y)]\le e^{\eta T}\lVert Y\rVert_\eta$ for $ds\otimes m$-a.e. $(s,y)$. This yields the uniform bound \eqref{eq:V_sup_moment_bound}. Finally, let $\widetilde{V}$ be as stated. Then $\widetilde{V}$ is predictable with $\lVert \widetilde{V}\rVert_\eta \leq \sup_{(t,x)}\E[\widetilde{V}_t(x)]<\infty$, so its equivalence class belongs to $\mathcal{H}_\eta^+$, and \eqref{eq:strong_solution} states that $\E[\lvert \Psi(\widetilde{V})(t,x)-\widetilde{V}_t(x)\rvert]=0$ for every $(t,x)$, whence $\lVert \Psi(\widetilde{V})-\widetilde{V}\rVert_\eta=0$. Thus $\widetilde{V}$ is a fixed point of $\Psi$ in $\mathcal{H}_\eta^+$, and the uniqueness in Theorem~\ref{thm:Psi_contraction} gives $\widetilde{V}=Y=V$ $\Prob\otimes dt\otimes m$-almost everywhere. For the pointwise assertion, fix $(t,x)\in[0,T]\times E$. Since $V=\widetilde{V}$ holds $\Prob\otimes dt\otimes m$-almost everywhere, the final statement of Lemma~\ref{lem:L1_estimate} yields $\Psi(V)(t,x)=\Psi(\widetilde{V})(t,x)$ almost surely. Combining this with the two solution identities $V_t(x)=\Psi(V)(t,x)$ and $\widetilde{V}_t(x)=\Psi(\widetilde{V})(t,x)$ gives $V_t(x)=\widetilde{V}_t(x)$ almost surely. Finally, for fixed $t$, Tonelli's theorem yields that
  \[
  \E\left[\int_E\mathbf{1}_{\{V_t(x)\neq\widetilde{V}_t(x)\}}\,m(dx)\right]=\int_E\Prob(V_t(x)\neq\widetilde{V}_t(x))\,m(dx)=0,
  \]
  so that, almost surely, $V_t=\widetilde{V}_t$ $m$-almost everywhere.
\end{proof}

\paragraph{Proof of Corollary~\ref{cor:activity_dichotomy}}
\begin{proof}
  Set $C_t:=\mu\bigl((I\cap(0,t])\times B\times\Rp\bigr)$ for $t\in[0,T]$ and $C:=C_T$, and recall from the proof of Proposition~\ref{prop:affine_compensator} the compensator identity $\E[\int H\,d\mu]=\E[\int H\,d\nu]$, valid in $[0,\infty]$ for every non-negative $\Pcal\otimes\Bcal(\Rp)$-measurable $H$.

  If $\ell((0,\infty))<\infty$, the identity applied to $H=\mathbf{1}_{I\times B\times\Rp}$ gives $\E[C]=\ell((0,\infty))\,\E[\xi]<\infty$, so that $C<\infty$ almost surely, which is the second claim.

  Suppose next that $\ell((0,\infty))=\infty$ and fix $n\in\N$. The process $t\mapsto C_t$ is increasing, integer-valued, and adapted, since the restriction of $\mu$ to $(0,t]\times E\times\Rp$ is a measurable function of the restriction of $N$ to $(0,t]\times E\times\Rp\times\Rp$ and of the predictable field $V$ on $[0,t]$. We therefore have that
  \[
  \rho_n:=\inf\lbrace t\in[0,T]: C_t\geq n\rbrace, \qquad(\inf\emptyset:=\infty),
  \]
  satisfies $\lbrace\rho_n\leq t\rbrace=\bigcap_{k\in\N}\lbrace C_{(t+1/k)\wedge T}\geq n\rbrace\in\Fcal_t^N$ by monotonicity of $C$ and right-continuity of the filtration, and hence $\rho_n$ is a stopping time. For every $s<\rho_n$, we have $C_s\leq n-1$. Lemma~\ref{lem:distinct_times} shows that at most one atom of $\mu$ lies at time $\rho_n$. Therefore,
  \[
  \mu\bigl((I\cap(0,\rho_n])\times B\times\Rp\bigr)\leq n
  \qquad\text{almost surely}.
  \]
  The integrand $\mathbf{1}_{(0,\rho_n]}(s)\,\mathbf{1}_{I}(s)\,\mathbf{1}_B(y)\,\mathbf{1}_{(1/j,\infty)}(z)$ is $\Pcal\otimes\Bcal(\Rp)$-measurable because $\mathbf{1}_{(0,\rho_n]}$ is left-continuous and adapted, so the compensator identity yields
  \[
  \ell((1/j,\infty))\;\E\left[\int_{I\cap(0,\rho_n]}\int_B\left(\lambda_0+\lambda_1V_s(y)\right)m(dy)\,ds\right] \leq n, \qquad j\in\N.
  \]
  Since $\ell((1/j,\infty))\uparrow\infty$, letting $j\to\infty$ forces
  \[
  \int_{I\cap(0,\rho_n]}\int_B(\lambda_0+\lambda_1V_s(y))\,m(dy)\,ds=0
  \qquad\text{almost surely}.
  \]
  On $\lbrace C<n\rbrace$, we have $\rho_n=\infty$, and the integral above equals $\xi$. Hence $\Prob(C<n,\xi>0)=0$ for every $n$. Taking the union over $n$ gives
  \[
  \Prob(C<\infty,\xi>0)=0,
  \]
  which proves the third claim.
  
  It remains to prove the first claim for an arbitrary $\ell$. Define
  \[
  g(t):=\int_{I\cap(0,t]}\int_B(\lambda_0+\lambda_1V_s(y))\,m(dy)\,ds,
  \]
  which is continuous, increasing, and adapted. For $\delta>0$, the time $\tau_\delta:=\inf\lbrace t\in[0,T]:g(t)\geq\delta\rbrace$, with $\inf\emptyset:=\infty$, is a stopping time and continuity gives $g(\tau_\delta\wedge T)\leq\delta$. Applying the compensator identity to the predictable integrand $\mathbf{1}_{(0,\tau_\delta]}(s)\,\mathbf{1}_I(s)\,\mathbf{1}_B(y)\,\mathbf{1}_{(1/j,\infty)}(z)$ gives
  \[
  \E\left[\mu\bigl((I\cap(0,\tau_\delta])\times B\times(1/j,\infty)\bigr)\right] \leq \ell((1/j,\infty))\,\delta.
  \]
  On $\lbrace\xi=0\rbrace$, we have $g\equiv0$ and hence $\tau_\delta=\infty$. It follows that
  \[
  \E\left[\mu(I\times B\times(1/j,\infty))\mathbf{1}_{\lbrace\xi=0\rbrace}\right]
  \leq\ell((1/j,\infty))\delta.
  \]
  Since $\ell((1/j,\infty))<\infty$, letting $\delta\downarrow0$ gives
  $\Prob(\xi=0,\mu(I\times B\times(1/j,\infty))>0)=0$. Letting $j\to\infty$ proves $\Prob(\xi=0,C>0)=0$, because $\mu$ does not charge $\lbrace z=0\rbrace$.
\end{proof}

\section{Spatial moment bounds and proof of Theorem~\ref{thm:holder_modification}}\label{app:spatial_regularity}

To bound higher order moments of $V_t(x)$, we rely on the following classical moment inequality for compensated integrals, often referred to as an inequality of Bichteler-Jacod type and due, in the present form, to Novikov \cite{Novikov1975}. See also \cite[Theorem~3.2]{MarinelliRockner2014} for a modern account and \cite{ChongKluppelberg2015} for closely related estimates in the ambit-field setting.

\begin{lemma}\label{lem:kunita}
  Let $p\geq2$, and let $\xi$ be an integer-valued random measure on $(0,T]\times E\times\Rp$ with predictable compensator $\zeta$. Assume that $\zeta$ has no predictable time atoms, i.e., for every predictable stopping time $\tau$,
  \[
  \zeta(\{\tau\}\times E\times\Rp)=0\quad\text{almost surely}.
  \]
  Let $H:\Omega\times(0,T]\times E\times\Rp\to\R$ be predictable and suppose that $\int|H|\,d\xi$ and $\int|H|\,d\zeta$ are almost surely finite. Then there exists a constant $C_p>0$, depending only on $p$, such that
  \[
  \E\left[\left\lvert \int H\,d\xi - \int H\,d\zeta\right\rvert^p\right] \leq C_p\left( \E\left[\left(\int H^2\,d\zeta\right)^{p/2}\right] + \E\left[\int \lvert H\rvert^p\,d\zeta\right] \right),
  \]
  holds with values in $[0,\infty]$.
\end{lemma}
\begin{proof}
  If the right-hand side is infinite, the claim is immediate, so assume that it is finite. The finite-valued c\`adl\`ag predictable increasing process
  \[
  A_t:=\int_{(0,t]\times E\times\Rp}|H|\,d\zeta
  \]
  is locally bounded. Its jumps can be exhausted by predictable stopping times, and the absence of predictable time atoms gives
  \[
  \Delta A_\tau=\int_{E\times\Rp}|H(\tau,y,z)|\,\zeta(\{\tau\},dy,dz)=0
  \]
  at every predictable stopping time $\tau$. Hence $A$ is continuous. Define
  \[
  \tau_n:=\inf\{t\in[0,T]:A_t\geq n\}\wedge T,
  \]
  with $\inf\varnothing:=T$. The stopping times $\tau_n$ are predictable, increase to $T$, and continuity gives $A_{\tau_n}\leq n$. The compensation identity therefore yields
  \[
  \E\left[\int_{(0,\tau_n]\times E\times\Rp}|H|\,d\xi\right]
  =\E\left[\int_{(0,\tau_n]\times E\times\Rp}|H|\,d\zeta\right]
  \leq n.
  \]
  Proposition~II.1.28 of \cite{JacodShiryaev2003} now identifies
  \[
  M_t:=\int_{(0,t]\times E\times\Rp}H\,d\xi
  -\int_{(0,t]\times E\times\Rp}H\,d\zeta,
  \]
  as a purely discontinuous local martingale, and each $M^{\tau_n}$ is an integrable martingale. If $\tau$ is predictable, the compensation identity applied after stopping gives
  \[
  \E\left[\mathbf{1}_{\{\tau\leq\tau_n\}}
  \int_{E\times\Rp}|H(\tau,y,z)|\,\xi(\{\tau\},dy,dz)\right]=0,
  \]
  because $\zeta(\{\tau\}\times E\times\Rp)=0$. Thus $M$ has no predictable jumps and is quasi-left-continuous. Apply the case $\alpha=2$ of \textup{(BJ)} in \cite[Theorem~3.2]{MarinelliRockner2014} to $M^{\tau_n}$. Fatou's lemma on the left and monotone convergence of the stopped compensator integrals on the right yield the stated estimate at $T$.
\end{proof}

\begin{proposition}\label{prop:p_moments}
  Let $p\geq 2$ and assume that $\overline{z}_p<\infty$ and that $k_2$ and $k_p$ are Borel measurable with $k_2,k_p\in L^1([0,T])$. Then
  \[
  \sup_{(t,x)\in[0,T]\times E}\E\left[V_t(x)^p\right]<\infty.
  \]
\end{proposition}
\begin{proof}
  Consider the Picard iteration $V^{(0)}:=\phi$ and $V^{(m+1)}:=\Psi(V^{(m)})$, defined pointwise by \eqref{eq:psi_operator}, and fix $\eta>0$ with $\overline{z}\lambda_1J_\eta(T)<1$. By \eqref{eq:Psi_bound} and induction, $B_1:=\sup_{m}\lVert V^{(m)}\rVert_\eta<\infty$, and as in the proof of Lemma~\ref{lem:L1_estimate} we obtain the uniform pointwise first-moment bound
  \[
  \E\left[V^{(m)}_t(x)\right] \leq \lVert\phi\rVert_\infty+\overline{z}\left(\lambda_0+\lambda_1e^{\eta T}B_1\right)\int_0^Tk(r)\,dr =: M_1, \qquad m\in\N,\ (t,x)\in[0,T]\times E.
  \]
  Set $a^{(m)}_s(y):=\lambda_0+\lambda_1V^{(m)}_s(y)$. The proofs of Proposition~\ref{prop:affine_compensator} and Corollary~\ref{cor:drift_martingale} only use the predictability and the uniform first-moment bound of the field appearing in the thinning indicator. Applying these to $V^{(m)}$, it follows that the accepted-jump measure $\mu_m$ associated with the indicator $\mathbf{1}_{\lbrace u\leq a^{(m)}_s(y)\rbrace}$ has compensator $\nu_m(ds,dy,dz) = a^{(m)}_s(y)\,ds\,m(dy)\,\ell(dz)$ and that, almost surely,
  \begin{equation}\label{eq:V_t_m+1_bound}
  \begin{aligned}
  V^{(m+1)}_t(x)={}&\phi_t(x)
  +\overline{z}\int_0^t\int_EK(t-s,x,y)\,a^{(m)}_s(y)\,m(dy)\,ds\\
  &+H_{t,x}\star(\mu_m-\nu_m)_t,
  \end{aligned}
  \end{equation}
  where $H_{t,x}(s,y,z):=\mathbf 1_{\{s<t\}}K(t-s,x,y)z$. Write $A_q:=\int_0^Tk_q(r)\,dr$ for $q\in\lbrace1,2,p\rbrace$. The compensator $\nu_m=a_s^{(m)}(y)\,ds\,m(dy)\,\ell(dz)$ has no predictable time atoms. Moreover, the uniform first-moment estimate and admissibility imply that the $\mu_m$- and $\nu_m$-integrals of $K(t-s,x,y)z$ are almost surely finite. Lemma~\ref{lem:kunita} therefore applies.

  We bound the three contributions in \eqref{eq:V_t_m+1_bound} separately. If $A_2=0$, the quadratic compensator integral below vanishes and we set $B_{2,p}=0$ in this case. If $A_2>0$, set $B_{2,p}:=A_2^{p/2-1}$ and apply H\"older's inequality to the finite measure $K(t-s,x,y)^2m(dy)ds$. Combining this with the analogous bound for $K(t-s,x,y)m(dy)ds$ and the scalar inequality $(b_1+b_2+b_3)^p\leq3^{p-1}(b_1^p+b_2^p+b_3^p)$ gives
  \[
  \begin{aligned}
  \E\left[\bigl(V^{(m+1)}_t(x)\bigr)^p\right] &\leq 3^{p-1}\biggl( \lVert\phi\rVert_\infty^p + \overline{z}^pA_1^{p-1}\int_0^t\int_EK(t-s,x,y)\,\E\left[\bigl(a^{(m)}_s(y)\bigr)^p\right]m(dy)\,ds \\
  &\qquad + C_p\,\overline{z}_2^{p/2}B_{2,p}\int_0^t\int_EK(t-s,x,y)^2\,\E\left[\bigl(a^{(m)}_s(y)\bigr)^{p/2}\right]m(dy)\,ds \\
  &\qquad + C_p\,\overline{z}_p\,A_p\left(\lambda_0+\lambda_1M_1\right) \biggr).
  \end{aligned}
  \]
  The assumptions $\overline z<\infty$ and $\overline z_p<\infty$ imply that $\overline z_q<\infty$ for every $q\in[1,p]$, in particular for the second and intermediate moments used above. Using $a^{p/2}\leq1+a^p$ and $\E[(a^{(m)}_s(y))^p]\leq2^{p-1}(\lambda_0^p+\lambda_1^p\E[(V^{(m)}_s(y))^p])$, we obtain constants $C_1,C_2>0$, depending only on $p$, $C_p$, the mark moments, $\lambda_0,\lambda_1$, $\lVert\phi\rVert_\infty$, $A_1,A_2,A_p$, and $M_1$, such that
  \begin{equation}\label{eq:p_moment_recursion}
  \E\left[\bigl(V^{(m+1)}_t(x)\bigr)^p\right] \leq C_1 + C_2\int_0^t\bigl(k(t-s)+k_2(t-s)\bigr)\,u_m(s)\,ds, \qquad (t,x)\in[0,T]\times E,
  \end{equation}
  whenever $u_m$ is a measurable function with $\E[(V^{(m)}_s(y))^p]\leq u_m(s)$ for $ds\otimes m$-a.e. $(s,y)$. Accordingly, define 
  \[
  u_0:=\lVert\phi\rVert_\infty^p, \quad u_{m+1}(t):=C_1+C_2\int_0^t(k+k_2)(t-s)\,u_m(s)\,ds,
  \]
  and choose $\eta'\geq0$ so large that $\theta:=C_2\int_0^Te^{-\eta' r}(k+k_2)(r)\,dr<1$, which is possible by dominated convergence since $k+k_2\in L^1([0,T])$. Setting 
  \[
  A:=(C_1\vee\lVert\phi\rVert_\infty^p)/(1-\theta),
  \] 
  it follows by induction that $u_m(t)\leq Ae^{\eta' t}$ for all $m$ and $t$. Indeed, $u_0\leq A$, and
  \[
  u_{m+1}(t)\leq C_1+C_2A\int_0^te^{\eta' s}(k+k_2)(t-s)\,ds \leq C_1+A\theta e^{\eta' t}\leq Ae^{\eta' t}.
  \]
  Hence $B_p:=\sup_m\sup_{(t,x)}\E[(V^{(m)}_t(x))^p]\leq Ae^{\eta' T}<\infty$. The Banach fixed point theorem yields that $V^{(m)}\to Y$ in $\mathcal{H}_\eta$ and thus it holds that $V^{(m)}_t(x)\to Y_t(x)$ in $L^1(\Prob)$ $dt\otimes m$-a.e. and in particular almost surely along a subsequence. Fatou's lemma yields $\E[Y_t(x)^p]\leq B_p$ $dt\otimes m$-a.e. Finally, since $V=\Psi(Y)$, one further application of \eqref{eq:p_moment_recursion}, with $Y$ in place of $V^{(m)}$ and the constant majorant $u\equiv B_p$, gives 
  \[
  \E[V_t(x)^p]\leq C_1+C_2B_p\int_0^T(k+k_2)(r)\,dr<\infty,
  \]
  for every $(t,x)\in[0,T]\times E$.
\end{proof}

\begin{proposition}\label{prop:increment_moments}
  Let the assumptions of Proposition~\ref{prop:p_moments} and \eqref{eq:phi_modulus} hold, and assume in addition that there exist $\beta\in(0,1]$ and non-negative functions $h_q\in L^q([0,T])$, $q\in\lbrace1,2,p\rbrace$, such that for all $x,x'\in E$,
  \begin{equation}\label{eq:K_Lq_modulus}
  \left(\int_E\lvert K(r,x,y)-K(r,x',y)\rvert^q\,m(dy)\right)^{1/q} \leq h_q(r)\,d_E(x,x')^\beta, \qquad r\in(0,T],\ q\in\lbrace1,2,p\rbrace.
  \end{equation}
  Then there exists a constant $C>0$ such that
  \[
  \sup_{t\in[0,T]}\E\left[\lvert V_t(x)-V_t(x')\rvert^p\right] \leq C\,d_E(x,x')^{\beta p}, \qquad x,x'\in E.
  \]
\end{proposition}
\begin{proof}
  Fix $t,x,x'$ and write 
  \[
  \begin{aligned}
  \Delta K(r,y)&:=K(r,x,y)-K(r,x',y), \\
  d&:=d_E(x,x'), \\
  a_s(y)&:=\lambda_0+\lambda_1V_s(y).
  \end{aligned}
  \]
  Set $S_q:=\sup_{(s,y)\in[0,T]\times E}\E[a_s(y)^q]$ for $q\in\lbrace1,p\rbrace$, and $S_q$ is finite due to \eqref{eq:V_sup_moment_bound} and Proposition~\ref{prop:p_moments}. Define $H_{t,x,x'}(s,y,z):=\mathbf 1_{\{s<t\}}\Delta K(t-s,y)z$. Subtracting the two instances of the compensated-noise representation \eqref{eq:drift_martingale_form} gives, almost surely,
  \[
  \begin{aligned}
  V_t(x)-V_t(x') &= \phi_t(x)-\phi_t(x') + \overline{z}\int_0^t\int_E\Delta K(t-s,y)\,a_s(y)\,m(dy)\,ds \\
  &\quad + H_{t,x,x'}\star(\mu-\nu)_t,
  \end{aligned}
  \]
  where all terms are almost surely finite. By \eqref{eq:K_Lq_modulus} with $q=1$, the measure $\lvert\Delta K(t-s,y)\rvert\,m(dy)\,ds$ on $(0,t)\times E$ has total mass at most $H_1d^\beta$ with $H_1:=\lVert h_1\rVert_{L^1([0,T])}$, so H\"older's inequality yields
  \[
  \begin{aligned}
  \E\left[\left(\int_0^t\int_E\lvert\Delta K(t-s,y)\rvert\,a_s(y)\,m(dy)\,ds\right)^p\right] &\leq \left(H_1d^\beta\right)^{p-1}S_p\int_0^t\int_E\lvert\Delta K(t-s,y)\rvert\,m(dy)\,ds \\
  &\leq H_1^p\,S_p\,d^{\beta p}.
  \end{aligned}
  \]
  The compensator $\nu=a_s(y)\,ds\,m(dy)\,\ell(dz)$ has no predictable time atoms. In addition, $|\Delta K|\leq K(\cdot,x,\cdot)+K(\cdot,x',\cdot)$ and the first-moment estimate show that the $\mu$- and $\nu$-integrals of $|\Delta K(t-s,y)|z$ are almost surely finite. Lemma~\ref{lem:kunita} therefore applies to $\xi=\mu$ and $H=\Delta K(t-s,y)z$ and gives the two terms
  \[
  \begin{aligned}
  \E\left[\left(\overline{z}_2\int_0^t\int_E\Delta K(t-s,y)^2\,a_s(y)\,m(dy)\,ds\right)^{p/2}\right] &\leq \overline{z}_2^{p/2}\left(H_2^2d^{2\beta}\right)^{p/2-1}\left(1+S_p\right)H_2^2\,d^{2\beta} \\
  &= \overline{z}_2^{p/2}H_2^p\left(1+S_p\right)d^{\beta p}
  \end{aligned}
  \]
  with $H_2:=\lVert h_2\rVert_{L^2([0,T])}$, where we used H\"older's inequality with respect to $\Delta K(t-s,y)^2\,m(dy)\,ds$ (of total mass at most $H_2^2d^{2\beta}$) and $a^{p/2}\leq1+a^p$, and
  \[
  \begin{aligned}
  \E\left[\int_{(0,t)\times E\times\Rp}\lvert\Delta K(t-s,y)\rvert^pz^p\,d\nu\right] &= \overline{z}_p\int_0^t\int_E\lvert\Delta K(t-s,y)\rvert^p\,\E\left[a_s(y)\right]m(dy)\,ds \\
  &\leq \overline{z}_p\,S_1\,\lVert h_p\rVert_{L^p([0,T])}^p\,d^{\beta p}.
  \end{aligned}
  \]
  Combining the three estimates with \eqref{eq:phi_modulus} and $(b_1+b_2+b_3)^p\leq3^{p-1}(b_1^p+b_2^p+b_3^p)$ proves the claim, with a constant $C$ depending only on $p$, $C_p$, $C_\phi$, $\overline{z},\overline{z}_2,\overline{z}_p$, $S_1,S_p$, $H_1,H_2$ and $\lVert h_p\rVert_{L^p}$.
\end{proof}

\printbibliography

\end{document}